\documentclass{amsart}
\usepackage{array}
\usepackage{tabularx}
\usepackage{amsmath, amssymb, amscd, calligra, mathrsfs}
\usepackage{mathrsfs}
\usepackage[all]{xy}
\usepackage{comment}
\usepackage[
pdftex,
bookmarks=false,
colorlinks=true,
debug=true,
pdfnewwindow=true]{hyperref}

\theoremstyle{definition}
\newtheorem{Def}{Definition}[section]
\newtheorem{Thm}[Def]{Theorem}
\newtheorem{Lem}[Def]{Lemma}
\newtheorem{Prop}[Def]{Proposition}
\newtheorem{Cor}[Def]{Corollary}
\newtheorem{Rem}[Def]{Remark}

\numberwithin{equation}{section}

\newcommand{\Z}{\mathbb{Z}}
\newcommand{\Q}{\mathbb{Q}}
\newcommand{\N}{\mathbb{N}}
\newcommand{\C}{\mathbb{C}}

\newcommand{\End}{\mathop{\mathrm{End}}\nolimits}

\newcommand{\Hom}{\mathop{\mathrm{Hom}}\nolimits}
\newcommand{\Stab}{\mathop{\mathrm{Stab}}\nolimits}

\newcommand{\codim}{\mathop{\mathrm{codim}}\nolimits}
\newcommand{\Gr}{\mathrm{Gr}}
\newcommand{\OO}{\mathbb{O}}
\newcommand{\Oo}{\mathcal{O}}

\newcommand{\Kk}{\mathcal{K}}
\newcommand{\Cc}{\mathcal{C}}
\newcommand{\bfa}{\mathbf{a}}
\newcommand{\diag}{\mathrm{diag}}
\newcommand{\Dom}{\mathbf{D}}
\newcommand{\SG}{\mathfrak{S}}
\newcommand{\Sp}{\mathtt{Sp}}
\newcommand{\Nn}{\mathcal{N}}
\newcommand{\Mm}{\mathcal{M}}
\newcommand{\T}{{}^{\mathsf{T}}\!}
\newcommand{\Irr}{\mathop{\mathsf{Irr}}\nolimits}
\newcommand{\IC}{\mathrm{IC}}
\newcommand{\sHom}{\mathscr{H}\text{\kern -3pt {\calligra\large om}}\,}
\newcommand{\Pp}{\mathcal{P}}
\newcommand{\KP}{\mathrm{KP}}
\newcommand{\Ee}{\mathcal{E}}

\newcommand{\LV}{\mathtt{LV}}
\newcommand{\Perv}{P}
\newcommand{\CC}{\mathfrak{C}}

\newcommand{\nc}{\newcommand}
\newcommand{\on}{\operatorname}
\nc{\ul}{\underline}
\nc{\wt}{\widetilde}
\nc{\Spec}{\mathop{\mathrm{Spec}}\nolimits}
\nc{\st}{{\mathsf{t}}}

\begin{document}

\title[Coherent IC-sheaves and dual canonical basis]
    {Coherent IC-sheaves on type $A_{n}$ affine Grassmannians and dual canonical basis
  of affine type $A_{1}$ 
	%$U_q^+(\widehat{\mathfrak{sl}}_2)$
    }

    
    \author{Michael Finkelberg}
 \address{M.F.:
  National Research University Higher School of Economics, Russian Federation,
  Department of Mathematics, 6 Usacheva st., Moscow 119048;
  Skolkovo Institute of Science and Technology;
  Institute for Information Transmission Problems}
 \email{fnklberg@gmail.com}
 
 \author{Ryo Fujita}
 \address{R.F.:
 Institut de Math\'{e}matiques de Jussieu-Paris Rive Gauche, IMJ-PRG,
 Universit\'{e} de Paris, B\^{a}timent Sophie Germain, F-75013, Paris, France}
 \email{ryo.fujita@imj-prg.fr}
   
    \begin{abstract}
      The convolution ring $K^{GL_n(\Oo)\rtimes\C^\times}(\Gr_{GL_n})$ was identified with
      a quantum unipotent cell of the loop group $LSL_2$ in~\cite{CW18}. We identify
      the basis formed by the classes of irreducible equivariant perverse coherent sheaves with
      the dual canonical basis of the quantum unipotent cell.
    \end{abstract}

    \maketitle

    \section{Introduction}

    \subsection{}
    The affine Grassmannian $\Gr_{GL_n}=GL_n(\Kk)/GL_n(\Oo)$ (where
    $\Kk=\C(\!(t)\!),\ \Oo=\C[\![t]\!]$),
    is a basic object of the geometric Langlands program.
    The convolution ring $K^{GL_n(\Oo)\rtimes\C^\times}(\Gr_{GL_n})$ (where $\C^\times$ acts by loop
    rotations) is a simplest example of the
    quantized $K$-theoretic Coulomb branch of a quiver gauge theory (for $\mathsf{A}_1$-quiver).
    The corresponding non-quantized $K$-theoretic Coulomb branch
    $K^{GL_n(\Oo)}(\Gr_{GL_n})$ is a commutative ring, whose spectrum is the trigonometric
    zastava space $^\dagger\!\overset{\circ}{Z}{}^n_{\mathfrak{sl}_2}$ of type $\mathsf{A}_1$ and degree $n$
    (alias moduli space of periodic $SU(2)$-monopoles of topological charge $n$).
    This space was thoroughly studied in yet another disguise in~\cite{GSV11},
    where its coordinate ring was equipped with a cluster structure.

    It is expected by physicists that all the $K$-theoretic Coulomb branches of gauge
    theories should carry a (generalized) cluster structure (for trigonometric
    zastava see~\cite{FKR18}). Moreover, it is expected that the quantized
    $K$-theoretic Coulomb branches should carry a quantum cluster structure.
    In the simplest example of $K^{GL_n(\Oo)\rtimes\C^\times}(\Gr_{GL_n})$ such a
    structure was exhibited in~\cite{CW18}. It was identified with a well known
    cluster structure on a quantum unipotent cell of the loop group $LSL_2$
    (in the non-quantized case, general trigonometric zastava spaces are identified with
    appropriate affine Richardson varieties in~\cite{FKR18}).

    Furthermore, all the cluster monomials of this cluster structure have a nice
    geometric meaning as classes in $K^{GL_n(\Oo)\rtimes\C^\times}(\Gr_{GL_n})$ of certain
    irreducible $GL_n(\Oo)\rtimes\C^\times$-equivariant perverse coherent sheaves
    on the affine Grassmannian $\Gr_{GL_n}$. The abelian monoidal category 
	$\Perv_{coh}^{GL_{n}(\Oo) \rtimes \C^{\times}}(\Gr_{GL_{n}})$	
	of perverse coherent sheaves was introduced in~\cite{BFM05}.
    Its $K$-ring coincides with $K^{GL_n(\Oo)\rtimes\C^\times}(\Gr_{GL_n})$, and it is equipped
    with a distinguished basis formed by the classes of IC-sheaves: irreducible equivariant
    perverse coherent sheaves. The problem of algebraic computation of this distinguished
    basis was standing ever since the appearance of~\cite{BFM05}, and the cluster monomials
    description of certain IC-classes given in~\cite{CW18} was a breakthrough in this direction.
    However, the IC-classes representable as cluster monomials only form a tip of the iceberg
    of all the IC-classes; namely, they are IC-extensions of certain equivariant {\em line} bundles
    on $GL_n(\Oo)$-orbits in $\Gr_{GL_n}$ (so this is similar to the lowest KL-cell in an
    affine Hecke algebra).

    Now the cluster monomials in a quantum unipotent cell of $LSL_2$ form a part of the
    dual canonical basis of the quantum group $U_q^+(\widehat{\mathfrak{sl}}_2)$ 
    (more precisely, of a certain localization of a subalgebra of its restricted dual) thanks to \cite{KKKO18}. So it is natural to
    expect that this dual canonical basis corresponds to the above distinguished basis formed
    by the IC-classes. This is indeed proved in the present paper.

    The dual canonical basis is characterized by two properties: (1) invariance with respect
    to a certain bar-involution; (2) the fact that the transformation matrix to the dual
    canonical basis from the dual PBW basis is identity modulo $q^{-1}$ (where
    $\Z[q^{\pm1}]=K_{\C^\times}(\on{pt})$). The corresponding bar-involution on
    $K^{GL_n(\Oo)\rtimes\C^\times}(\Gr_{GL_n})$ (fixing IC-classes) was introduced in~\cite{CW18}.
    The dual PBW basis corresponds to certain convolutions of line bundles on the first
    minuscule orbit in $\Gr_{GL_n}$. The analogue of property (2) for the usual {\em constructible}
    IC-sheaves is very deep (it boils down to the Riemann-Weil conjecture proved by Deligne).
    In the {\em coherent} setting of equivariant sheaves on nilpotent cone
    $\Perv_{coh}^{G\times\C^\times}(\Nn_{\mathfrak g})$ a similar property was proved by
    Bezrukavnikov
    in~\cite{Bez06} by making use of his coherent-constructible correspondence and reducing
    to the above result of Deligne. In this setting the role of dual PBW basis is played by
    the classes of Andersen-Jantzen sheaves (pushforwards of the dominant line bundles from
    the Springer resolution of the nilpotent cone).

    We are able to check the property (2) by reducing it to Bezrukavnikov's theory for
    $\Perv_{coh}^{GL_d\times\C^\times}(\Nn_{\mathfrak{gl}_d})$. Namely, we consider a closed
    subvariety $\Gr_n^d\subset\Gr_{GL_n}$ formed by all the sublattices of codimension $d$
    in the standard lattice $\Oo^n\subset\Kk^n$. Then we have a natural smooth morphism
    of stacks \begin{equation*}\psi\colon [(GL_n(\Oo)\rtimes\C^\times)\backslash\Gr_n^d]\to
    [(GL_d\times\C^\times)\backslash\Nn_{\mathfrak{gl}_d}],\end{equation*} and $\psi^*$ takes coherent
    IC-sheaves to IC-sheaves, and the Andersen-Jantzen sheaves to the appropriate $d$-fold
    convolutions of line bundles on $\Gr_n^1$.

    The appearance of the {\em dual} canonical basis is natural from yet another point of view.
    According to~\cite{FT18}, $K^{GL_n(\Oo)\rtimes\C^\times}(\Gr_{GL_n})$ is the homomorphic image
    of a certain integral form (i.e.\ a $\Q[q^{\pm1}]$-subalgebra) of a shifted quantum affine
    algebra of type $\mathsf{A}_1$. This integral form is spanned by the {\em dual}
    Poincar\'e-Birkhoff-Witt-Drinfeld basis.

    Finally, we should note that the problem of algebraic characterization of the IC-basis
    of $K^{G(\Oo)\rtimes\C^\times}(\Gr_G)$ makes sense for arbitrary reductive group $G$, and we
    have no clue how to approach it for $G\ne GL_n$.

    \subsection{Acknowledgments} We are grateful to R.~Bezrukavnikov, S.~Cautis, S.~Kato,
    A.~Tsymbaliuk and H.~Williams for the useful discussions. Moreover, the key idea to apply
    the relation between affine Grassmannians of type $A$ and nilpotent cones of type $A$
    (going back to~\cite{L81}) to computation of classes of coherent IC-sheaves
    is due to R.~Bezrukavnikov.

    M.F.\ was partially funded within the framework of the HSE University Basic Research Program
    and the Russian Academic Excellence Project `5-100'.
	R.F.\ was supported by Grant-in-Aid for JSPS Research Fellow (No.~18J10669) 
	and in part by Kyoto Top Global University program.  

	\subsection{Overall convention} 
	A variety always means a complex algebraic variety. 
	Let $X$ be a variety equipped with an action of a complex algebraic group $G$. 
	For a (closed) point $x \in X$, we denote by $\Stab_{G} x$ the stabilizer of $x$ in $G$.
	We denote by $D^{G}_{coh}(X)$ (resp.~$D^{G}_{qcoh}(X)$) 
	the derived category of bounded (resp.~unbounded) $G$-equivariant complexes of sheaves on $X$ 
	whose cohomologies are coherent (resp.~quasi-coherent).
	We denote by $\mathbb{D} := \mathbb{R} \sHom(-, \omega_{X})$ the 
	Grothendieck-Serre duality functor on $D^{G}_{coh}(X)$,
	where $\omega_{X}$ is a $G$-equivariant dualizing complex on $X$.   
	For a group automorphism 
	$\rho$ of $G$, we denote by $X^{\rho}$ the same variety $X$ with 
	a new $G$-action obtained by twisting the original $G$-action by $\rho$.
	For an object $\mathcal{F} \in D^{G}_{coh}(X)$, we denote by $\mathcal{F}^{\rho}$
	the sheaf obtained by twisting the $G$-equivariant structure of $\mathcal{F}$ by $\rho$.
	Then $\mathcal{F}^{\rho}$ is an object of $D^{G}_{coh}(X^{\rho})$.

	For an abelian category $\mathcal{A}$, we denote by $\Irr \mathcal{A}$
	the set of isomorphism classes of simple objects of $\mathcal{A}$. 
	We abbreviate $\Irr G := \Irr \mathrm{Rep}(G)$, where $\mathrm{Rep}(G)$
	is the category of finite-dimensional algebraic representations of $G$ over $\C$.  

\section{Quantum unipotent cell of $LSL_{2}$}

%%%%%%%%%%%%%%%%%%%%%%%%%%%%%%%%%%%%%%%%%%%%%%%%%%%%%%%%%%%%%%%%%%%%%%%%%%%%%%%%%%%%%%%%%%%%%%%%%%%%%%%%%%%%%%%%%%

\subsection{Quantum algebras of type $\mathsf{A}_{1}^{(1)}$}

Let $
\left( \begin{array}{cc}
a_{00} & a_{01} \\
a_{10} & a_{11}
\end{array}
\right)
=
\left( \begin{array}{cc}
2 & -2 \\
-2 & 2
\end{array}
\right)
$
be the generalized Cartan matrix of type $\mathsf{A}^{(1)}_{1}$ and
$\mathsf{Q} := \Z \alpha_{0} \oplus \Z \alpha_{1}$ be the root lattice
($\alpha_{0}, \alpha_{1}$ are the simple roots).
We define a symmetric bilinear form $(-,-)$ on $\mathsf{Q}$ by 
$(\alpha_{i}, \alpha_{j}) = a_{ij}$ for $i, j \in \{0,1\}$.
We set $\mathsf{Q}^{+} := \N \alpha_{0} + \N \alpha_{1} \subset \mathsf{Q}$. 

Let $U^{+} \equiv U_{q}^{+}(\widehat{\mathfrak{sl}}_{2})$ be the $\Q(q)$-algebra generated by the
two generators $\{ e_{0}, e_{1}\}$ satisfying the quantum Serre relation
$$
e_{i}^{3}e_{j} - (q^{2}+q^{-2})e_{i}^{2}e_{j}e_{i} 
+ (q^{2}+q^{-2}) e_{i}e_{j}e^{2}_{i} - e_{j}e_{i}^{3} = 0
$$
for $ \{ i,j \} = \{ 0,1\}.$
The algebra $U^{+}$ is the positive (or the upper triangular) part 
of the quantized enveloping algebra $U_{q}(\widehat{\mathfrak{sl}}_{2})$. 
We define the weight grading 
$U^{+} = \bigoplus_{\beta \in \mathsf{Q}^{+}} U_{\beta}^{+}$
by setting $\deg e_{i} := \alpha_{i}$.

We equip the tensor square $U^{+} \otimes_{\Q(q)} U^{+}$ with 
a $\Q(q)$-algebra structure by
$$
(x_{1} \otimes x_{2} ) \cdot (y_{1} \otimes y_{2})
:= q^{(\beta_{2}, \gamma_{1})} x_{1}y_{1} \otimes x_{2} y_{2},
$$ where $
x_{i} \in U_{\beta_{i}}^{+}, y_{i} \in U_{\gamma_{i}}^{+}
$.
Let $r\colon  U^{+} \to U^{+} \otimes_{\Q(q)} U^{+}$ be a $\Q(q)$-algebra
homomorphism given by
$$
r(e_{i}) := e_{i} \otimes 1 + 1 \otimes e_{i}
$$
for $i=0,1$. 
Let $A =  
\bigoplus_{\beta \in \mathsf{Q}^{+}} A_{\beta}
=
\bigoplus_{\beta \in \mathsf{Q}^{+}}
\Hom_{\Q(q)}(U^{+}_{\beta}, 
\Q(q))$ be the restricted dual 
of $U^{+}$ with a $\Q(q)$-algebra structure
given by the dual of $r$. 

%%%%%%%%%%%%%%%%%%%%%%%%%%%%%%%%%%%%%%%%%%%%%%%%%%%%%%%%%%%%%%%%%%%%%%%%%%%%%%%%%%%%%%

%\subsection{Bilinear form}

Following Lusztig~\cite[Proposition 1.2.3]{Lus93},
we define a
nondegenerate symmetric $\Q(q)$-bilinear form 
$(-,-)_{L}$ on $U^{+}$ by 
$$
(1,1)_{L} = 1, \quad (e_{i}, e_{j})_{L} = (1-q^{-2})^{-1} \delta_{ij}, \quad
(x, yz)_{L} = (r(x), y \otimes z)_{L},
$$
where $(x \otimes y , z \otimes w)_{L} := (x,z)_{L} \cdot 
(y, w)_{L}.$
%On the other hand,
%Kashiwara's nondegenerate 
%symmetric $\Q(q)$-bilinear form $(-,-)_{K}$ on $U^{+}$ is given by
%$$
%(1,1)_{K} = 1, \quad (e_{i}, e_{j})_{K} = \delta_{ij}, \quad
%(x, yz)_{K} = (r(x), y \otimes z)_{K}.
%$$
%For $x \in U^{+}_{\beta}, y \in U^{+}_{\gamma}$, 
%we can prove by induction that
%$(x,y)_{L}=(x,y)_{K} = 0$ if $\beta \neq \gamma$,
%and that
%$
%(x,y)_{K} = (1-q^{-2})^{\mathsf{ht}\, \beta} (x, y)_{L}
%$
%if $\beta = \gamma$.
The bilinear form $(-,-)_{L}$ induces 
a $\Q(q)$-algebra isomorphism
$\psi_{L}\colon U^{+} \cong A$  
defined by
$
\langle \psi_{L}(x), y \rangle = (x, y)_{L},
$
where 
$\langle -, - \rangle\colon A \times U^{+} \to \Q(q)$
is the natural pairing.

%%%%%%%%%%%%%%%%%%%%%%%%%%%%%%%%%%%%%%%%%%%%%%%%%%%%%%%%%%%%%%%%%%%%%%%%%%%%%%%%%%%%%%%%%%%%%%%%%%

%\subsection{Bar involutions}

We define an algebra involution $\mathtt{b}$ of $U^{+}$
by
$$
\mathtt{b}(q) = q^{-1}, \quad
\mathtt{b}(e_{i}) = e_{i}.
$$
Let $\mathtt{b}^{*}$
denote the $\Q$-linear involution of $A$
defined as the dual of $\mathtt{b}$, i.e.~for $\theta \in A, x \in U^{+}$, we define
$$
\langle \mathtt{b}^{*}(\theta), x \rangle :=
\overline{\langle \theta, \mathtt{b}(x) \rangle}
%\quad
%(x, \mathtt{b}^{*}_{K}(y))_{K} = \overline{(
%\mathtt{b}(x), y)_{K}},
$$
where $\overline{f(q)} := f(q^{-1})$ for $f(q) \in \Q(q)$.
Then for $\theta_{i} \in A_{\beta_i} \, (i = 1,2)$, we have 
$
\mathtt{b}^{*}(\theta_1 \theta_2) = q^{-(\beta_1, \beta_2)}
\mathtt{b}^{*}(\theta_2) \mathtt{b}^{*}(\theta_1).
$
%Moreover, we have 
%\begin{equation}
%\label{Eq:bars}
%\mathtt{b}^{*}(x) = (-1)^{\height \beta} q^{-\frac{1}{2}(\beta, \beta)
%- \height \beta} \mathtt{b}(\sigma(x)), 
%\quad
%\mathtt{b}_{K}^{*}(x) = q^{-\frac{1}{2}(\beta, \beta)
%+ \height \beta} \mathtt{b}(\sigma(x)),
%\end{equation}
%where $x \in U^{+}_{\beta}$
%and $\sigma$ is a $\Q(q)$-algebra anti-involution on $U^{+}$
%such that $\sigma (e_{i}) = e_{i}$
%(cf.~\cite[Proposition 3.2]{Kim12}).

%\begin{Rem}
%In the convention of \cite{GLS13}, 
%$\mathtt{b}(x) = \overline{x}$, 
%$\sigma(x) = x^{*}$. 
%\end{Rem}

%%%%%%%%%%%%%%%%%%%%%%%%%%%%%%%%%%%%%%%%%%%%%%%%%%%%%%%%%%%%%%%%%%%%%%%%%%%%%%%%%%%%%%%%%%%%

\subsection{Quantum unipotent subgroup}

We fix $n \in \N$ throughout this paper. 
Let 
$w_{n} = 
s_{i_1} \cdots s_{i_{2n}}
:= (s_{0}s_{1})^{n}$
be an element of the Weyl group of type
$\mathsf{A}^{(1)}_{1}$ of length $2n$, 
where $s_{i}$ is the simple reflection associated with the index $i \in \{ 0,1\}$. 
For each $1 \le k \le 2n$, we define the positive root $\beta_{k}$ by 
$$
\beta_{k} := s_{i_1} \cdots s_{i_{k-1}} (\alpha_{i_k})
= k \alpha_{0} + (k-1)\alpha_{1}. 
$$
The roots $\beta_{1}, \ldots, \beta_{2n}$ are all the positive roots $\alpha$ 
such that $w_{n}^{-1} (\alpha) < 0$. 
Define the corresponding root vectors by
$$
E(\beta_{k}) := T_{i_1} \circ \cdots \circ T_{i_{k-1}} (e_{i_k})
$$
for $1 \le k \le 2n$, where
$T_{i}$ denotes Lusztig's symmetry 
($=T^{\prime}_{i, -1}$ in Lusztig's notation, see~\cite[37.1]{Lus93} for the definition).
Let $U^{+}_{n} \subset U^{+}$ be the $\Q(q)$-subalgebra
generated by $\{ E(\beta_{k}) \mid 1 \le k \le 2n\}$.
The subalgebra 
$A_{n} := \psi_{L}(U^{+}_{n})$
of $A$ is
called the {\em quantum unipotent subgroup} associated with 
$w_{n}$. 
Both algebras $U^{+}_{n}$ and $A_{n}$ inherit the $\mathsf{Q}^{+}$-gradings 
from $U^{+}$ and $A$:
$$
U^{+}_{n} = \bigoplus_{\beta \in \mathsf{Q}^{+}} (U^{+}_{n})_{\beta},
\quad 
A_{n} = \bigoplus_{\beta \in \mathsf{Q}^{+}} (A_{n})_{\beta}. 
$$

%%%%%%%%%%%%%%%%%%%%%%%%%%%%%%%%%%%%%%%%%%%%%%%%%%%%%%%%%%%%%%%%%%%%%%%%%%%%%%%%%%%%%%%%%%%%%%%%%%%%%%%%%

\subsection{PBW and dual PBW bases}

For an element $\beta \in \mathsf{Q}^{+}$, 
we set  
$$
\KP_{n}(\beta) := \{ 
\bfa = (a_{1}, \ldots, a_{2n}) \in (\Z_{\ge 0})^{2n}
\mid
a_{1} \beta_{1} + \cdots + a_{2n} \beta_{2n} = \beta
\}. 
$$
For each $\bfa =(a_1, a_2, \ldots, a_{2n}) \in \KP_{n}(\beta)$,
we define the corresponding PBW element
as the product of $q$-divided powers by 
$$
E(\bfa) := 
E(\beta_{1})^{(a_{1})}
E(\beta_{2})^{(a_2)}
\cdots
E(\beta_{2n})^{(a_{2n})},
$$
where $x^{(\ell)} := x^{\ell} / (\prod_{i=1}^{\ell} [i]_{q}), \, [i]_{q} := \frac{q^{i} - q^{-i}}{q - q^{-1}}$ as usual.     
Then the set $\{ E(\bfa) \mid \bfa \in \KP_{n}(\beta) \}$
forms a $\Q(q)$-basis of $(U^{+}_{n})_{\beta}$. 
It is known (cf.~\cite[Proposition 38.2.3]{Lus93}) that we have
$(E(\bfa), E(\mathbf{b}))_{L} 
= 0$ if
$\bfa \neq \mathbf{b}$ and 
$$
(E(\bfa), E(\bfa))_{L} = \prod_{k=1}^{2n}
\prod_{j=1}^{a_k} (1-q^{-2j})^{-1}.
$$
For each $\bfa \in \KP_{n}(\beta)$,
we define the dual PBW element in $A_{n}$ by
$$
E^{*}(\bfa) := \frac{\psi_{L}(E(\bfa))}{
(E(\bfa), E(\bfa))_{L}
}. 
$$
By construction, the set $\{E^{*}(\bfa) \mid \bfa \in \KP_{n}(\beta) \}$ forms 
a $\Q(q)$-basis of $(A_{n})_{\beta}$ dual to the basis $\{ E(\bfa) \mid \bfa \in \KP_{n}(\beta) \}$ of $(U_{n}^{+})_{\beta}$.

The dual PBW element $E^{*}(\bfa)$ can be written simply as a product of the dual root vectors 
$E^{*}(\beta_{k}) = (1-q^{-2})\psi_{L}(E(\beta_{k}))$:
\begin{equation} \label{Eq:dual PBW}
E^{*}(\bfa) = q^{- \sum_{k=1}^{2n} a_{k}(a_{k}-1) /2} E^{*}(\beta_{1})^{a_{1}} E^{*}(\beta_{2})^{a_{2}} \cdots E^{*}(\beta_{2n})^{a_{2n}}.
\end{equation}

%%%%%%%%%%%%%%%%%%%%%%%%%%%%%%%%%%%%%%%%%%%%%%%%%%%%%%%%%%%%%%%%%%%%%%%%%%%%%%%%%%%%%%%%%%%%%%%%%%%%%%%%

\subsection{Dual canonical basis} 

Let 
$\mathbf{B} \subset U^{+}$ be the canonical basis of $U^{+}$ 
constructed in \cite[Part II]{Lus93}.
It is characterized up to sign as the set of elements $b \in U^{+}$ 
satisfying $\mathtt{b}(b) = b$ and $(b, b)_{L} \in 1 + q^{-1} \Z[\![ q^{-1} ]\!]$.
The dual canonical basis $\mathbf{B}^{*}$ is defined 
as the basis of $A$ dual to the canonical basis $\mathbf{B}$.
By definition, each element of $\mathbf{B}^{*}$ is fixed by the bar involution $\mathtt{b}^{*}$.

The following theorem due to Kimura~\cite{Kim12}
claims that the dual canonical basis 
is compatible with the quantum unipotent subgroup $A_{n}$
and characterized by using the dual PBW basis.  
\begin{Thm}[{\cite[Theorem 4.29]{Kim12}}] \label{Thm:Kimura}
For each $\beta \in \mathsf{Q}^{+}$,
there exists a unique $\Q(q)$-basis
$\mathscr{B}_{n}(\beta) = \{ B^{*}(\bfa) \mid \bfa \in \KP_{n}(\beta) \}$
of $(A_{n})_{\beta}$ characterized by the following properties:
\begin{itemize}
\item[(1)] $\mathtt{b}^{*}(
B^{*}(\bfa)) = B^{*}(\bfa)$;
\item[(2)] $B^{*}(\bfa) 
\in  E^{*}(\bfa) + \sum_{\bfa^{\prime}
\in \KP_{n}(\beta)} q^{-1} \Z[q^{-1}] E^{*}(\bfa^{\prime}).$
\end{itemize}
Moreover we have
$\mathscr{B}_{n}(\beta) = \mathbf{B}^{*} \cap (A_{n})_{\beta}.$ 
\end{Thm}

We refer to the basis $\bigsqcup_{\beta \in \mathsf{Q}^{+}} \mathscr{B}_{n}(\beta)$
as the dual canonical basis of $A_{n}$. 
We denote by
$
A_{n, \Z}
$
the integral form of $A_{n}$, i.e.~the $\Z[q^{\pm 1}]$-subalgebra of $A_{n}$ spanned by the dual canonical basis (or the dual PBW basis).

Let $\mathsf{P} = \Z \varpi_{0} \oplus \Z \varpi_{1}$
be the weight lattice of type $\mathsf{A}^{(1)}_{1}$ with 
$\varpi_{i}$ being the $i$-th fundamental weight ($i \in \{0,1\}$).
For a dominant weight $\lambda \in \mathsf{P}^{+} := \N \varpi_{0} + \N \varpi_{1}$ and 
Weyl group elements $u$ and $v$ satisfying $u(\lambda) - v(\lambda) \in \mathsf{Q}^{+}$, 
we denote by 
$D_{u(\lambda), v(\lambda)}$ 
the corresponding {\em quantum unipotent minor} (see~\cite[Section~5.2]{GLS13} for the definition).
This is an element of $\mathscr{B}_{n}(u(\lambda) - v(\lambda))$ (see \cite[Proposition~6.3]{GLS13}).  
For each $1 \le k \le 2n$, we set 
$w_{\le k} := s_{i_1} s_{i_2} \cdots s_{i_k}$
and 
$w_{< k} := s_{i_1} s_{i_2} \cdots s_{i_{k-1}}$.
For each $1 \le b \le d \le 2n$ with $d - b \in 2 \Z$, we define
$
D[b, d] := D_{w_{< b} \varpi_{i_b}, 
w_{\le d} \varpi_{i_d}}.
$
By \cite[Proposition~7.4]{GLS13}, we have
$
E^{*}(\beta_{k}) =
D[k, k]
$
for each $1 \le k \le 2n$.

%%%%%%%%%%%%%%%%%%%%%%%%%%%%%%%%%%%%%%%%%%%%%%%%%%%%%%%%%%%%%%%%%%%%%%%%%%%%%%%%%%%%%%%%%%%%%

\subsection{Quantum cluster structure}

Let $\mathscr{A}_{n}$ be the quantum cluster algebra (over $\Z[q^{\pm 1/2}]$)
associated with the initial quantum seed $\mathscr{S} = ((X_{1}, \ldots, X_{2n}), \widetilde{B}, \Lambda)$
defined as follows (see~\cite{BZ05} for the generalities of quantum cluster algebras).  
The exchange matrix $\widetilde{B}$ is the $2n \times (2n-2)$-matrix 
given by 
$$
\widetilde{B} :=
\left(
\begin{array}{cccccccc}
0 & -2 & 1 & & & & &\\
2 & 0 & -2 & 1 & & & &\\
-1 &  2 & 0 & -2 & 1 & & &\\
& -1 & 2 & 0 & -2 & 1 & &\\  
& & \ddots & \ddots & \ddots & \ddots & \ddots &\\
& & & -1 & 2 & 0 & -2 & 1\\
& & & & -1 & 2 & 0 & -2\\  
&&&&&-1 & 2 & 0 \\ \hline
&&&&&&-1 & 2  \\
&&&&&&& -1
\end{array}
\right),
$$
where all blank entries are $0$. 
The skew-symmetric 
$2n \times 2n$-matrix $\Lambda=(\Lambda_{k\ell})$ is given by
$
\Lambda_{k \ell} := 2\lceil k/2 \rceil (\lfloor k/2 \rfloor - \lfloor \ell/2 \rfloor)   
$
for any $1 \le k < \ell \le 2n$ (cf.~\cite[Lemma~6.4]{CW18}). 
The quantum cluster algebra $\mathscr{A}_{n}$ is a $\Z[q^{\pm 1/2}]$-subalgebra of the based quantum torus 
$\Z[q^{\pm 1/2}] \langle X_{k}^{\pm 1} \, (1 \le k \le 2n) \mid X_{k} X_{\ell} = q^{\Lambda_{k \ell}} X_{\ell} X_{k} \rangle$.   

For each $1 \le b \le d \le 2n$ with $d - b \in 2\Z$,
we consider the following
normalized element in $\Q(q^{1/2}) \otimes_{\Q(q)} A_{n}$: 
$$
\widetilde{D}[b,d]
:= q^{(\beta, \beta)/4}D[b, d],
$$
where $\beta = 
w_{< b} \varpi_{i_b} 
- w_{\le d} \varpi_{i_d}
$
is the weight of $D[b, d]$.
More generally, for any quantum unipotent minor $D_{u(\lambda), v(\lambda)}$, we define
$\widetilde{D}_{u(\lambda), v(\lambda)} := q^{(\beta, \beta)/4}D_{u(\lambda), v(\lambda)}$ with 
$\beta = u(\lambda) - v(\lambda) \in \mathsf{Q}^{+}$.
%\begin{Rem}
%Here we use the notation of \cite{CW18}.
%In the notation of \cite{GLS13}, we have $D[b, d] = D(b^{-}, d)$.
%\end{Rem}

\begin{Thm}[{\cite[Theorem 12.3]{GLS13}, \cite[Corollary 11.2.8]{KKKO18}}] \label{Thm:GLS}
There is a unique isomorphism 
$$
\mathscr{A}_{n}
\xrightarrow{\sim}
\Z[q^{\pm 1/2}] \otimes_{\Z[q^{\pm 1}]} A_{n, \Z}
$$
under which the initial cluster variable $X_{k}$ corresponds to
$$
\begin{cases}
\widetilde{D}[1, k] & \text{if $k$ is odd;} \\
\widetilde{D}[2, k] & \text{if $k$ is even,} 
\end{cases}
$$
for each $1 \le k \le 2n$.
Moreover the quantum unipotent minor $\widetilde{D}[b, d]$ is the image of a cluster variable
for any $1 \le b \le d \le 2n$ with $d-b \in 2 \Z$. 
\end{Thm}

Henceforth, we will identify the quantum unipotent subgroup 
$\Z[q^{\pm 1/2}] \otimes_{\Z[q^{\pm 1}]} A_{n, \Z}$
with the quantum cluster algebra $\mathscr{A}_{n}$
via the isomorphism in Theorem~\ref{Thm:GLS}. 

%%%%%%%%%%%%%%%%%%%%%%%%%%%%%%%%%%%%%%%%%%%%%%%%%%%%%%%%%%%%%%%%%%%%%%%%%%%%%%%%%%%%%%%%%%%%%%%%%%%%%%%%%%%%%%%%%%%%%%%

\subsection{Berenstein-Zelevinsky's bar involution}
 \label{Ssec:BZ involution}

When $x, y$ are $q$-commuting with each other, say
$xy = q^{m}yx$ for some $m \in \Z$, we
write $x \odot y := q^{- m/2} xy$.
Note that we have $x \odot y = y \odot x$.  
Following Berenstein-Zelevinsky~\cite{BZ05}, let us consider
the algebra anti-involution $\iota$ of 
the quantum cluster algebra $\mathscr{A}_{n}$
defined by
$$\iota(q)=q^{-1}, \quad
\iota (X_{k}) = X_{k},$$
for all $
1 \le k \le 2n$. 
If $x, y \in \mathscr{A}_{n}$ are
$q$-commuting with each other and both are fixed by $\iota$,
the element $x \odot y = y \odot x$ 
is also fixed by $\iota$. 
In particular, any cluster variables and hence any 
cluster monomials are fixed by $\iota$.  

\begin{Lem}[cf.~{\cite[Remark 7.23]{KO17}}]
If $x \in \mathscr{A}_{n}$ 
is of weight $\beta$, we have
$\iota(x) = q^{(\beta, \beta)/2} \mathtt{b}^{*}(x)$.
In particular, $x$ is fixed by $ \mathtt{b}^{*}$
if and only if the rescaled element $q^{(\beta, \beta)/4} x$
is fixed by $\iota$.
\end{Lem}

For any $\beta \in \mathsf{Q}^{+}$ and $\bfa \in \KP_{n}(\beta)$, we consider the rescaled elements
$$
\widetilde{B}(\bfa) := 
q^{(\beta, \beta)/4} B^{*}(\bfa), 
\quad
\widetilde{E}(\bfa) := 
q^{(\beta, \beta)/4} E^{*}(\bfa).
$$
Theorem~\ref{Thm:Kimura} yields the following
characterization of the rescaled dual canonical basis.
\begin{Cor}
 \label{Cor:characterization}
The basis $\{ \widetilde{B}(\bfa) \mid
\bfa \in \KP_{n}(\beta) \}$ of $(\mathscr{A}_{n})_{\beta}$ is characterized by the following 
properties:
\begin{enumerate}
\item 
$\iota(\widetilde{B}(\bfa)) = 
\widetilde{B}(\bfa)$;
\item
$\widetilde{B}(\bfa) 
\in  \widetilde{E}(\bfa) + \sum_{\bfa^{\prime} \in \KP_{n}(\beta)} q^{-1} \Z[q^{-1}] \widetilde{E}(\bfa^{\prime}).$
\end{enumerate}
\end{Cor}

Note that in particular the rescaled dual root vector is 
$$\widetilde{E}(\beta_k) = q^{1/2}E^{*}(\beta_k) = \widetilde{D}[k,k]$$
for each $1 \le k \le 2n$.
In terms of the rescaled elements,
the expression~(\ref{Eq:dual PBW}) is rewritten as 
\begin{equation} \label{Eq:rescaled dual PBW}
\widetilde{E}(\bfa) = q^{\sum_{k < \ell} a_{k}a_{\ell}} \widetilde{E}(\beta_{1})^{a_{1}} 
\widetilde{E}(\beta_{2})^{a_2}
\cdots \widetilde{E}(\beta_{2n})^{a_{2n}}
\end{equation}
for each $\bfa \in \KP_{n}(\beta)$.

%%%%%%%%%%%%%%%%%%%%%%%%%%%%%%%%%%%%%%%%%%%%%%%%%%%%%%%%%%%%%%%%%%%%%%%%%%%%%%%%%%%%%%%%%%%%%%%%%%%%%%%%%%%%%

\subsection{Localization}

Note that the frozen variables of the quantum cluster algebra $\mathscr{A}_{n}$ 
are the following two elements:
$$
D_{0} := X_{2n-1} = \widetilde{D}_{\varpi_{0}, w_{n} \varpi_{0}},
\quad
D_{1} := X_{2n} = \widetilde{D}_{\varpi_{1}, w_{n} \varpi_{1}}.
$$

\begin{Prop}[cf.~{\cite[Proposition 4.2]{KO17}}]
\label{Prop:frozen}
%Let $\mathsf{P}^{+} := \N \varpi_{0} + \N \varpi_{1}$ be the set of dominant weights.
For each $\lambda \in \mathsf{P}^{+}$, the unipotent quantum minor 
$D_{\lambda, w_{n} \lambda}$ is $q$-central in $A_{n}$.
More precisely, for any homogeneous element $x \in (\mathscr{A}_{n})_{\beta}$ of weight $\beta \in \mathsf{Q}^{+}$,
we have 
$
D_{\lambda, w_{n} \lambda} x = q^{-(\lambda + w_{n} \lambda, \beta)} x D_{\lambda, w_{n} \lambda}.
$
Moreover, for $\lambda, \lambda^{\prime} \in \mathsf{P}^{+}$, we have
$
\widetilde{D}_{\lambda, w_{n} \lambda} \odot \widetilde{D}_{\lambda^{\prime}, w_{n} \lambda^{\prime}} 
= \widetilde{D}_{\lambda + \lambda^{\prime}, w_{n}(\lambda + \lambda^{\prime})}
$
in $\mathscr{A}_{n}$. In particular, we have
$
\widetilde{D}_{\lambda, w_{n}\lambda} = D_{0}^{\odot \ell_{0}} \odot D_{1}^{\odot \ell_{1}}
$
for $\lambda = \ell_{0} \varpi_{0} + \ell_{1} \varpi_{1} \in \mathsf{P}^{+}$.
\end{Prop}

By Proposition~\ref{Prop:frozen},
the set $\mathcal{D}_{n} := \{ q^{m/2} D_{0}^{\ell_{0}} D_{1}^{\ell_{1}} \mid m \in \Z,
\ell_{i} \in \N \}$ is an Ore set of the algebra $\mathscr{A}_{n}$.
The localized algebra
$\mathscr{A}_{n}^{loc} := \mathscr{A}_{n}[\mathcal{D}_{n}^{-1}]$ is (isomorphic to) the 
{\em quantum unipotent cell} associated with $w_{n}$ (cf.~\cite[Section 4]{KO17}).

\begin{Prop}[cf.~{\cite[Proposition 4.5]{KO17}}]
The set 
$$
\widetilde{\mathscr{B}}_{n}^{loc} :=
\{ \widetilde{B}(\bfa) \odot D_{0}^{\odot (-\ell_{0})} \odot D_{1}^{\odot (-\ell_{1})} \mid
\bfa \in \N^{2n}, \ell_{i} \in \N \}
$$
forms a $\Z[q^{\pm 1/2}]$-basis of the quantum unipotent cell $\mathscr{A}_{n}^{loc}$.
\end{Prop}

Note that each element in $\widetilde{\mathscr{B}}_{n}^{loc}$ is fixed by the bar involution $\iota$.

%%%%%%%%%%%%%%%%%%%%%%%%%%%%%%%%%%%%%%%%%%%%%%%%%%%%%%%%%%%%%%%%%%%%%%%%%%%%%%%%%%%%%%%%%%%%%%%%%%%%%%%%%%%%%%
%%%%%%%%%%%%%%%%%%%%%%%%%%%%%%%%%%%%%%%%%%%%%%%%%%%%%%%%%%%%%%%%%%%%%%%%%%%%%%%%%%%%%%%%%%%%%%%%%%%%%%%%%%%%%%
\section{Perverse coherent sheaves on type $A$ affine Grassmannian}

%%%%%%%%%%%%%%%%%%%%%%%%%%%%%%%%%%%%%%%%%%%%%%%%%%%%%%%%%%%%%%%%%%%%%%%%%%%%%%%%%%%%%%%%%%%%%%%%%%%%%%%

\subsection{Affine Grassmannian}

Let $T_{n} \subset GL_{n}(\C)$ be the maximal torus 
consisting of diagonal matrices and
$P^{\vee} := \Hom(\C^{\times}, T_{n})$
be the coweight lattice. 
We make the standard identification $P^{\vee} = \Z^{n}$ under which  
the element $\nu = (\nu_{1}, \ldots, \nu_{n}) \in \Z^{n}$ 
corresponds to the $1$-parameter group $a \mapsto \diag(a^{\nu_1}, \ldots, a^{\nu_n})$.
The weight lattice $P = \Hom_{\Z}(P^{\vee}, \Z)$ is also identified with $\Z^{n}$ 
via the standard pairing $\langle - , - \rangle : \Z^{n} \times \Z^{n} \to \Z$ given by
$\langle \nu, \mu \rangle = \nu_{1} \mu_{1} + \cdots + \nu_{n} \mu_{n}$.
We say that an element $\nu = (\nu_{1}, \ldots, \nu_{n}) \in \Z^{n}$ 
is dominant if $\nu_{1} \ge \cdots \ge \nu_{n}$. Write $P^{\vee}_{+}$ for 
the set of dominant coweights.
For each $1 \le k \le n$, we define 
$$
\omega_{k} := (\overbrace{1, \ldots, 1}^{k}, 0, \ldots, 0) \in \Z^{n},
$$  
which is regarded as the $k$-th fundamental  weight or coweight.

Let 
$\Kk := \C(\!( t )\!) \supset \Oo := \C[\![ t ]\!]$
be the field of formal Laurent series and 
its subring of formal power series. 
We consider the affine Grassmannian of $GL_{n}$:
$$
\Gr_{GL_{n}} = GL_{n}(\Kk) / GL_{n}(\Oo) = (GL_{n}(\Kk)\rtimes \C^{\times}) / (GL_{n}(\Oo)\rtimes \C^{\times}),
$$
where $\C^{\times}$ denotes the $4$-fold cover of the standard loop rotation.
More precisely, we have $(1, a) \cdot (g(t), 1) = (g(a^{4}t), a)$ in $GL_{n}(\Oo) \rtimes \C^{\times}$ 
for $g(t) \in GL_{n}(\Oo),\ a \in \C^{\times}$. 
The affine Grassmannian $\Gr_{GL_{n}}$ decomposes into the union of 
$GL_{n}(\Oo)$-orbits ($= GL_{n}(\Oo) \rtimes \C^{\times}$-orbits):
$$
\Gr_{GL_{n}} = \bigsqcup_{\nu \in P^{\vee}_{+}} \Gr_{GL_{n}}^{\nu},
$$
where $\Gr_{GL_{n}}^{\nu}$ denotes the orbit of 
$[t^{\nu}] \in \Gr_{GL_{n}}$,
$t^{\nu} := \mathrm{diag}(t^{\nu_{1}}, \ldots, t^{\nu_{n}}) \in GL_{n}(\Kk)$.
For each $\nu = (\nu_{1}, \ldots, \nu_{n}) \in P^{\vee}_{+}$,
we have $\dim \Gr_{GL_{n}}^{\nu} = \sum_{k=1}^{n} (n+1 -2k)\nu_{k}$.

The closure $\overline{\Gr}{}_{GL_{n}}^{\nu}$ of the orbit $\Gr_{GL_{n}}^{\nu}$ is called the {\it Schubert variety}. 
Let us consider the derived category 
$D_{coh}^{GL_{n}(\Oo) \rtimes \C^{\times}}(\overline{\Gr}{}_{GL_{n}}^{\nu})$ 
of bounded $GL_{n}(\Oo)\rtimes \C^{\times}$-equivariant complexes of sheaves on the reduced scheme 
$(\overline{\Gr}{}_{GL_{n}}^{\nu})^{red}$ with coherent cohomologies, 
formally supported in cohomological degrees $\frac{1}{2} \dim \Gr_{GL_{n}}^{\nu} + \Z$ by convention.

The connected components of $\Gr_{GL_{n}}$ are labeled by $\Z$.
For each $d \in \Z$, the $d$-th connected component $\Gr_{GL_{n}}^{(d)}$ is
the union of $GL_{n}(\Oo)$-orbits $\Gr_{GL_n}^{\nu}$ with 
$d = \nu_{1}+ \cdots + \nu_{n}$.
Note that the parity of $\dim \Gr_{GL_n}^{\nu}$ is constant
on each connected component.    
Therefore we can define 
$$D_{coh}^{GL_{n}(\Oo)\rtimes \C^{\times}}(\Gr_{GL_{n}}) 
:= \varinjlim_{\nu} D_{coh}^{GL_{n}(\Oo) \rtimes \C^{\times}}(\overline{\Gr}{}_{GL_{n}}^{\nu}).$$

%Let us consider the bounded derived category 
%$$D_{coh}^{GL_{n}(\Oo)\rtimes \C^{\times}}(\Gr_{GL_{n}})
%= \bigoplus_{d \in \Z} D_{coh}^{GL_{n}(\Oo)\rtimes \C^{\times}}\left(\Gr_{GL_{n}}^{(d)}\right)
%$$
%of $GL_{n}(\Oo)\rtimes \C^{\times}$-equivariant coherent sheaves on $\Gr_{GL_n}$. 
%By convention, the $d$-th block 
%$D_{coh}^{GL_{n}(\Oo)\rtimes \C^{\times}}(\Gr_{GL_{n}}^{(d)})$ 
%consists of complexes cohomologically supported in (formal) degrees $\Z + \frac{1}{2} d(n-1)$.
%By definition, each object $\mathcal{F} \in D^{GL_{n}(\Oo) \rtimes \C^{\times}}_{coh}(\Gr_{GL_{n}})$
%is supported on a certain Schubert variety $\overline{\Gr}{}_{GL_n}^{\nu}$.  

\subsection{Convolution product}

For any objects $\mathcal{F}, \mathcal{G} \in D_{coh}^{GL_{n}(\Oo)\rtimes \C^{\times}}(\Gr_{GL_{n}})$,
we can define 
their {\em convolution product} $\mathcal{F} * \mathcal{G} \in D_{coh}^{GL_{n}(\Oo)\rtimes \C^{\times}}(\Gr_{GL_{n}})$ by
$$
\mathcal{F} * \mathcal{G} := \overline{m}_{*}( \mathcal{F} \, \wt{\boxtimes} \, \mathcal{G}), 
$$ 
where $\overline{m} : (GL_{n}(\Kk) \rtimes \C^{\times}) \times^{(GL_{n}(\Oo) \rtimes \C^{\times})} \Gr_{GL_{n}} \to \Gr_{GL_{n}}$ is the multiplication map.
Here the sheaf $\mathcal{F} \, \wt{\boxtimes} \, \mathcal{G}$ is defined by the property
$\overline{q}{}^{*} (\mathcal{F} \, \wt{\boxtimes} \, \mathcal{G})  \cong
\overline{p}{}^{*} (\mathcal{F} \, {\boxtimes} \, \mathcal{G})$,
where $\overline{p}$ and $\overline{q}$ are the natural projections: 
$$
\Gr_{GL_{n}} \times \Gr_{GL_{n}} \xleftarrow{\overline{p}} 
(GL_{n}(\Kk) \rtimes \C^{\times} ) \times \Gr_{GL_{n}} \xrightarrow{\overline{q}} 
(GL_{n}(\Kk) \rtimes \C^{\times}) \times^{GL_{n}(\Oo) \rtimes \C^{\times}} \Gr_{GL_{n}}.   
$$ 
If $\mathcal{F}$ (resp.~$\mathcal{G}$) is supported on $\overline{\Gr}{}_{GL_n}^{\nu}$ 
(resp.~$\overline{\Gr}{}_{GL_n}^{\nu^{\prime}}$), the sheaf $\mathcal{F} \, \wt{\boxtimes} \, \mathcal{G}$
is supported on the finite-dimensional {\em convolution variety} 
$$\overline{\Gr}{}_{GL_n}^{\nu}
\wt{\times} \overline{\Gr}{}_{GL_n}^{\nu^{\prime}} \subset GL_{n}(\Kk) \times^{GL_{n}(\Oo)} \Gr_{GL_{n}}$$
and the convolution product $\mathcal{F} * \mathcal{G}$ is supported on 
$\overline{\Gr}{}_{GL_n}^{\nu + \nu^{\prime}}$.

The convolution product $*$ equips the equivariant $K$-group 
$K^{GL_{n}(\Oo) \rtimes \C^{\times}}(\Gr_{GL_{n}})$ with a structure of an
associative $\Z[q^{\pm 1/2}]$-algebra,   
where $q^{m/2} \in K^{GL_{n}(\Oo) \rtimes \C^{\times}}(\mathrm{pt})$
denotes the class of the pull-back of the $1$-dimensional $\C^{\times}$-module $\C_m$
of weight $m$ along the natural projection $GL_{n}(\Oo) \rtimes \C^{\times} \to \C^{\times}$.  

We use the notation $\{ m/2 \}$ to denote the $\C^{\times}$-equivariant twist $-\otimes \C_{-m}$. 
Thus, for an object $\mathcal{F} \in D_{coh}^{GL_{n}(\Oo)\rtimes \C^{\times}}(\Gr_{GL_{n}})$, 
we have $[\mathcal{F} \{ m/2\}] = q^{-m/2} [\mathcal{F}]$.
On the other hand, we denote by 
$[m/2]$ the cohomological degree shift by $m/2 \in\frac12\Z$. 
It will be convenient to use the notation 
$\langle m/2 \rangle := [m/2]\{-m/2\}$
for the simultaneous shift and twist by $m/2 \in\frac12\Z$.
This is the same notation as in~\cite{CW18}.

%For brevity of notation, we write
%$$
%\KK_{n} := \Q(q^{1/2})\otimes_{\Z[q^{\pm 1/2}]} 
%K^{GL_{n}(\Oo)\rtimes \C^{\times}}(\Gr_{GL_{n}}).
%$$
%This is an associative $\Q(q^{1/2})$-algebra.

%%%%%%%%%%%%%%%%%%%%%%%%%%%%%%%%%%%%%%%%%%%%%%%%%%%%%%%%%%%%%%%%%%%%%%%%%%%%%%%%%%%%%%%%%%%%%%%%%%%%%%%%%%%%%%%

\subsection{Perverse coherent sheaves}

\begin{Def}
A $GL_{n}(\Oo) \rtimes \C^{\times}$-equivariant {\em perverse coherent sheaf} on $\Gr_{GL_n}$ is 
an object $\mathcal{F} \in D_{coh}^{GL_{n}(\Oo)\rtimes \C^{\times}}(\Gr_{GL_{n}})$ such that
for every orbit $i_{\nu} \colon  \Gr_{GL_n}^{\nu} \hookrightarrow \Gr_{GL_n}$
\begin{itemize}
\item[(1)] $i_{\nu}^{*}\mathcal{F} \in D_{qcoh}^{GL_{n}(\Oo)\rtimes \C^{\times}}(\Gr_{GL_{n}}^{\nu})$
is supported in degrees $\le -\frac{1}{2} \dim \Gr_{GL_{n}}^{\nu}$;
\item[(2)] $i_{\nu}^{!}\mathcal{F} \in D_{qcoh}^{GL_{n}(\Oo)\rtimes \C^{\times}}(\Gr_{GL_{n}}^{\nu})$
is supported in degrees $\ge -\frac{1}{2} \dim \Gr_{GL_{n}}^{\nu}$.
\end{itemize}
We denote by $\Perv_{coh}^{GL_{n}(\Oo) \rtimes \C^{\times}}(\Gr_{GL_{n}}) \subset D_{coh}^{GL_{n}(\Oo)\rtimes \C^{\times}}(\Gr_{GL_{n}})$ 
the full subcategory
of perverse coherent sheaves.
\end{Def}

The category $\Perv_{coh}^{GL_{n}(\Oo) \rtimes \C^{\times}}(\Gr_{GL_{n}})$ can be
obtained as the core of a finite-length $t$-structure (called the {\em perverse $t$-structure})
of the category $D_{coh}^{GL_{n}(\Oo)\rtimes \C^{\times}}(\Gr_{GL_{n}})$
(cf.~\cite{AB10}).
The convolution product $*$ preserves the category $\Perv_{coh}^{GL_{n}(\Oo) \rtimes \C^{\times}}(\Gr_{GL_{n}})$
and
the operation $(\mathcal{F}, \mathcal{G}) \mapsto \mathcal{F} * \mathcal{G}$ is bi-exact
(cf.~\cite{BFM05}).
Thus the equivariant $K$-group 
$
K^{GL_{n}(\Oo)\rtimes \C^{\times}}(\Gr_{GL_{n}}) = 
K(\Perv_{coh}^{GL_{n}(\Oo) \rtimes \C^{\times}}(\Gr_{GL_{n}}))
$
becomes an algebra with a canonical $\Z$-basis 
formed by the classes of simple perverse coherent sheaves.

We say that $(\nu, \mu) \in 
P^{\vee} \times P$ is a dominant pair  
if $\nu \in P^{\vee}$ is a 
dominant coweight
and $\mu \in P$ is dominant 
with respect to 
the Levi quotient of
the stabilizer subgroup
$\Stab_{GL_{n}(\Oo)}[t^{\nu}]$. 
%The set $\Dom_{n} \subset 
%P^{\vee} \times P$ of 
%dominant pairs forms a complete 
%system of representatives of the quotient $(P^{\vee} \times P) /\SG_{n}$,
%where $\SG_{n}$ is the symmetric group of degree $n$ 
%($=$ the Weyl group of $GL_{n}$).
More explicitly, the set $\Dom_{n}$ of dominant pairs is given by
$$
\Dom_{n} = \left \{
(\nu, \mu) \in P^{\vee}_{+} \times P \; \middle|
\begin{array}{l}
\nu = (\nu_1, \ldots, \nu_n), \mu = (\mu_1, \ldots, \mu_{n}), \\
\text{$\mu_{k} \ge \mu_{k+1}$ whenever $\nu_{k} = \nu_{k+1}$}
\end{array}
\right\}.
$$
 
To each dominant pair $(\nu, \mu) \in \Dom_{n}$, 
we associate a simple perverse coherent sheaf $\Pp_{\nu, \mu}$
in the following way.
Note that the group
$$
\Stab_{GL_{n}(\Oo)}^{red}[t^{\nu}] := \{ g \in GL_{n}(\C) \mid g \cdot t^{\nu} = t^{\nu} \cdot g \} 
\cong \textstyle \prod_{k} GL_{m_k}(\C)
$$
is a Levi subgroup of $\Stab_{GL_{n}(\Oo)}[t^{\nu}]$,
where $m_{k}$ is the multiplicity of $k$ in the sequence $\nu \in \Z^{n}$.
Then the group
$$
\Stab_{GL_{n}(\Oo)}^{red}[t^{\nu}] \rtimes \C^{\times} =
\Stab_{GL_{n}(\Oo)}^{red}[t^{\nu}] \times \C^{\times}
$$
is a Levi subgroup of $\Stab_{GL_{n}(\Oo) \rtimes \C^{\times}} [t^{\nu}]$.
Thus we can identify
$$
\Irr \Stab_{GL_{n}(\Oo) \rtimes \C^{\times}} [t^{\nu}]
= 
\Irr (\Stab_{GL_{n}(\Oo)}^{red} [t^{\nu}] \times \C^{\times}) 
\supset 
\Irr \Stab_{GL_{n}(\Oo)}^{red} [t^{\nu}],
$$
where the set $\Irr  \Stab_{GL_{n}(\Oo)}^{red} [t^{\nu}]$ is regarded as a 
subset of $\Irr (\Stab_{GL_{n}(\Oo)}^{red} [t^{\nu}] \times \C^{\times})$
consisting of representations with the trivial $\C^{\times}$-actions.
Let $\mathcal{V}_{\mu}$ denote the simple 
$GL_{n}(\Oo) \rtimes \C^{\times}$-equivariant vector bundle on $\Gr_{GL_n}^{\nu}$
whose fiber at $[t^{\nu}]$ is isomorphic to $V_{\mu}$ 
as a representation of $\Stab_{GL_{n}(\Oo) \rtimes \C^{\times}}[t^{\nu}]$. 
We define the simple perverse coherent sheaf $\Pp_{\nu, \mu}$ as
the following (coherent) $\IC$-extension (cf.~\cite[Theorem 4.2]{AB10})
$$
\Pp_{\nu, \mu} := (i_{\nu})_{!*} \mathcal{V}_{\mu} \langle \dim \Gr_{GL_{n}}^{\nu} /2 \rangle \{ - \langle \nu, \mu \rangle \},   
$$  
where $i_{\nu}\colon \Gr_{GL_{n}}^{\nu} \hookrightarrow \overline{\Gr}{}_{GL_n}^{\nu}$ is the inclusion. 

Since each simple perverse coherent sheaf is isomorphic to an $\IC$-extension of a simple vector bundle 
on some $GL_{n}(\Oo) \rtimes \C^{\times}$-orbit (cf.~\cite[Proposition 4.11]{AB10}),
we have a bijection
$$
\Dom_{n} \times \Z 
\overset{1:1}{\longleftrightarrow}
\Irr \Perv_{coh}^{GL_{n}(\Oo) \rtimes \C^{\times}}(\Gr_{GL_n});
\quad
((\nu, \mu), m)
\longleftrightarrow
\Pp_{\nu, \mu}\{m/2\}.
$$ 
In particular, the set 
$$\mathscr{P}_{n} :=  \{ [\Pp_{\nu, \mu}] \mid (\nu, \mu) \in \Dom_{n}\}$$
forms a $\Z[q^{\pm 1/2}]$-basis of the convolution ring $K^{GL_{n}(\Oo)\rtimes \C^{\times}}(\Gr_{GL_{n}})$.

%%%%%%%%%%%%%%%%%%%%%%%%%%%%%%%%%%%%%%%%%%%%%%%%%%%%%%%%%%%%%%%%%%%%%%%%%%%%%%%%%%%%%%%%%%%%%%%%%%%%%%%
\subsection{Lattice description}

Recall that the affine Grassmannian $\Gr_{GL_n}$ can be interpreted as 
the moduli space of $\Oo$-lattices $L$ in $\Kk^{n}$.
Let $L_{0} := \Oo^{n} \subset \Kk^{n}$ be the standard $\Oo$-lattice.
A coset $[g(t)] \in \Gr_{GL_n} = GL_{n}(\Kk) / GL_{n}(\Oo)$ corresponds to a lattice $L = g(t)L_{0}$.
%We define the positive part of $\Gr_{GL_n}$ by 
%$$
%\Gr_{GL_n}^{\ge 0} := \bigsqcup_{\nu \in P^{\vee}_{+, \ge 0}} \Gr_{GL_n}^{\nu}
%$$ 
%where $P^{\vee}_{+, \ge 0} := \{ 
%\nu = (\nu_{1}, \ldots, \nu_{n}) \in P^{\vee}_{+} \mid \nu_{n} \ge 0\}$.
Then for each $\nu \in P^{\vee}_{+}$ with $\nu_{n} \ge 0$, we have 
$$
\Gr_{GL_n}^{\nu} = \{ L \subset L_{0} \mid \text{$t|_{L_{0}/L}$ is nilpotent of type $\nu$}\}.
$$
In particular, when $\nu = \omega_{k}$, we get
$$
\Gr_{GL_n}^{k} := \Gr_{GL_n}^{\omega_{k}} = \{ L \overset{k}{\subset} L_{0} \mid tL_{0} \subset L \},
$$ 
where $L \overset{k}{\subset} L_{0}$ indicates that $\dim (L_{0}/L) = k$.
In particular, $\Gr_{GL_n}^{k} = \overline{\Gr}{}_{GL_n}^{k} $ 
is isomorphic to the usual Grassmannian $\Gr(k,n)$ of $k$-dimensional subspaces in $\C^{n}$.
 
For each $1 \le k \le n$ and $\ell \in \Z$, we put
$\Pp_{k, \ell} := \Pp_{\omega_{k}, \ell \omega_{k}}$ for simplicity. 
Using the above description of $\Gr_{GL_{n}}^{k}$,
we see
$$
\Pp_{k, \ell} = i_{\omega_k *}\left( 
\Oo_{\Gr_{GL_n}^{k}} \otimes \mathrm{det} (L_{0}/L)^{\ell} 
\right)
 \langle k(n-k)/2 \rangle \{-k \ell \},
$$
where we denote by $L_{0}/L$ the vector bundle on $\Gr_{GL_{n}}^{k}$
whose fiber at $L$ is 
equal to $L_{0}/L$ by an abuse of notation.

%%%%%%%%%%%%%%%%%%%%%%%%%%%%%%%%%%%%%%%%%%%%%%%%%%%%%%%%%%%%%%%%%%%%%%%%%%%%%%%%%%%%%%%%%%%%%%%%%%%%%%%%%%%%%%%5

\subsection{Cautis-Williams' monoidal categorification theorem}
\label{Ssec:CW}

In \cite{CW18}, Cautis and Williams
proved that, for a general complex reductive group $G$,
the category of 
$G(\Oo) \rtimes \C^{\times}$-equivariant perverse coherent sheaves 
is a {\em rigid} monoidal category, i.e.~every object $\mathcal{F}$ 
has its left and right duals.
Moreover, they also proved the existence of a 
{\em system of renormalized $r$-matrices} 
(originated in the settings of the quiver Hecke algebras and the quantum affine algebras,
see~\cite{KKKO18} for instance), 
which informally encodes some information about how the category fails to be 
a braided tensor category.
Using these facts, it was successfully proved in the case $G=GL_{n}$
that the monoidal category $\Perv_{coh}^{GL_{n}(\Oo) \rtimes \C^{\times}}(\Gr_{GL_{n}})$ 
categorifies the quantum unipotent cell $\mathscr{A}^{loc}_{n}$.
More precisely, we have:     
 
\begin{Thm}[Cautis-Williams \cite{CW18}]
\label{Thm:CW}
There exists an isomorphism
of $\Z[q^{\pm 1/2}]$-algebras
$$
\Phi \colon  \mathscr{A}^{loc}_{n} \xrightarrow{\sim} K^{GL_{n}(\Oo)\rtimes \C^{\times}}(\Gr_{GL_{n}}),
$$
which sends
each cluster monomial to the class of a simple perverse coherent sheaf. 
Moreover,
for each $1 \le b \le d \le 2n$ with $d-b \in 2\Z$, we have
$
\Phi(\widetilde{D}[b,d]) = [\Pp_{k, \ell}] 
$
with $$
k = 1+ \frac{1}{2}(d-b), \quad \ell = n+1 - \frac{1}{2}(b+d). 
$$
In particular, we have 
\begin{equation} \label{Eq:E=P}
\Phi(\widetilde{E}(\beta_{k})) = \Phi(\wt{D}[k,k]) = [\Pp_{1, n+1 - k}]
\end{equation}
for each $1 \le k \le 2n$.
\end{Thm}

The bar involution $\iota$ of $\mathscr{A}^{loc}_{n}$ was also categorified in \cite[Section 6.2]{CW18}. 
Let $\sigma_{1}$ be the group involution of $GL_{n}(\Oo) \rtimes \C^{\times}$ given by 
$(g(t), a) \mapsto (\T g(t)^{-1}, a)$, where $\T (-)$ denotes the transpose of matrices. 
Then the morphism 
$$\eta : \Gr_{GL_n} \to GL_{n}(\Kk) \times^{GL_{n}(\Oo)} (\Gr_{GL_n})^{\sigma_{1}} = 
(GL_{n}(\Kk) \rtimes \C^{\times}) \times^{(GL_{n}(\Oo) \rtimes \C^{\times})} (\Gr_{GL_n})^{\sigma_{1}}$$
defined by $\eta([g(t)]) := [g(t), [\T g(t)]]$ becomes $GL_{n}(\Oo) \rtimes \C^{\times}$-equivariant. 
We define an involutive auto-equivalence $\iota$ on $D^{GL_{n}(\Oo) \rtimes \C^{\times}}_{coh}(\Gr_{GL_n})$ by 
$$
\iota(\mathcal{F}) := \mathbb{D} \circ \mathbb{L} \circ \eta^{*} (\Oo_{\Gr_{GL_n}} \, \wt{\boxtimes} \, \mathcal{F}^{\sigma_{1}}), 
$$  
where $\mathbb{D}$ is the Grothendieck-Serre duality functor and $\mathbb{L}$ is an auto-equivalence 
on $D^{GL_{n}(\Oo) \rtimes \C^{\times}}_{coh}(\Gr_{GL_n})$ which, on the $d$-th component 
$\Gr_{GL_{n}}^{(d)}$, acts by tensoring with 
$$
\Oo_{\Gr_{GL_n}}(-n) \otimes \mathrm{det}(L_0 / t L_0)^{d}\{d(n-d)\}.
$$   

\begin{Thm}[{\cite[Corollary 6.24 \& 6.25]{CW18}}]
The involution $\iota$ is contravariant
with respect to both convolution product $*$ and $\mathrm{Hom}$, 
preserves the category of perverse coherent sheaves and satisfies
$\iota(\Pp_{\nu, \mu} \{m/2\}) \cong \Pp_{\nu, \mu} \{ -m/2 \}$
for any $(\nu, \mu) \in \Dom_{n}$ and $m/2 \in \frac12\Z$.
Therefore, for any $\xi \in \mathscr{A}^{loc}_{n}$, we have
$
\Phi(\iota \xi) = \iota \Phi(\xi).
$
\end{Thm}

Note that the basis
$\mathscr{P}_{n}$ of $K^{GL_{n}(\Oo) \rtimes \C^{\times}}(\Gr_{GL_n})$ is nothing but the subset 
formed by the classes of $\iota$-selfdual simple perverse coherent sheaves.

\section{Comparison with nilpotent cones of type $A$}

\subsection{Main result}

The main theorem of this paper is the following.
\begin{Thm}
\label{Thm:0}
Under Cautis-Williams' isomorphism $\Phi \colon  \mathscr{A}^{loc}_{n} \cong K^{GL_{n}(\Oo)\rtimes \C^{\times}}(\Gr_{GL_{n}})$ in
Theorem~\ref{Thm:CW},
the dual canonical basis $\widetilde{\mathscr{B}}^{loc}_{n}$ of $\mathscr{A}^{loc}_{n}$
bijectively corresponds to the basis $\mathscr{P}_{n}$ of $K^{GL_{n}(\Oo)\rtimes \C^{\times}}(\Gr_{GL_{n}})$ 
formed by the classes of $\iota$-selfdual simple perverse coherent sheaves.
\end{Thm}

By Theorem~\ref{Thm:CW}, we have $\Phi(D_{0}) = [\Pp_{n, 1}]$ and 
$\Phi(D_{1}) = [\Pp_{n, 0}]$. 
Since both $\Pp_{n, 0}$ and $\Pp_{n, 1}$ are invertible objects 
of the monoidal category $\Perv_{coh}^{GL_{n}(\Oo) \rtimes \C^{\times}}(\Gr_{GL_{n}})$,
the operations $- \odot [\Pp_{n, 0}]^{\pm 1}$ and $- \odot [\Pp_{n, 1}]^{\pm 1}$
induce the self-bijections of the set $\mathscr{P}_{n}$.
Therefore, to verify Theorem~\ref{Thm:0}, it suffices to 
prove the following simpler assertion:

\begin{Thm}
\label{Thm:1}
We have
$
\Phi( \widetilde{B}(\bfa)) \in \mathscr{P}_{n}
$
for any $\bfa \in \N^{2n}$.
\end{Thm}

A proof will be given in the end of Section~\ref{Ssec:comparison}. 

\begin{Rem}
When $n=2$, Theorem~\ref{Thm:0} can be verified directly
by using an explicit computation of the dual canonical basis of the quantum unipotent group $A_{2}$
due to Lampe~\cite{Lam14}.
\end{Rem}

%%%%%%%%%%%%%%%%%%%%%%%%%%%%%%%%%%%%%%%%%%%%%%%%%%%%%%%%%%%%%%%%%%%%%%%%%%%%%%%%%%%%%%%%%%%%%%%%%%%%%%%%%%%%%%%%%%%%%%%%%%%5

\subsection{Perverse coherent sheaves on the nilpotent cone}

Fix $d \in \N$. 
Let $$\mathcal{N}^{d} :=\{ x \in \End(\C^{d}) \mid x^{d} = 0 \}$$
be the nilpotent cone of $\mathfrak{gl}_{d}(\C)$. 
A left action of the group $GL_{d}(\C) \times \C^{\times}$ on $\mathcal{N}^{d}$
is given by $(g, a) \cdot x = a^{-4} \mathrm{Ad}(g) x$.
%\footnote{
%The convention of the $\C^{\times}$-action has been changed from the previous version. 
%}
The equivariant $K$-group $K^{GL_{d}(\C) \times \C^{\times}}(\mathcal{N}^{d})$ 
is a module over $\Z[q^{\pm 1/2}]$, 
where $q^{m/2} \in K^{GL_{d}(\C) \times \C^{\times}}(\mathrm{pt})$
denotes the class of the pull-back of the $1$-dimensional $\C^{\times}$-module $\C_m$
of weight $m$ along the natural projection $GL_{d}(\C) \times \C^{\times} \to \C^{\times}$.   

Recall that the nilpotent cone $\mathcal{N}^{d}$ has a finite number of $GL_{d}(\C)$-orbits
($= GL_{d}(\C) \times \C^{\times}$-orbits)
which are parametrized by the set $P(d)$
of partitions of $d$. 
The orbit $\OO_{\nu}$ labelled by a partition
$\nu =(\nu_{1} \ge \nu_{2} \ge \cdots ) \in P(d)$
consists of nilpotent matrices of Jordan type $\nu$,
i.e.~whose Jordan normal form is 
$$J_{\nu} := J_{\nu_{1}} \oplus J_{\nu_{2}} \oplus \cdots, \quad
\text{where $
J_{m} := \left(
\begin{array}{ccccc}
0 & & & &  \\
1 & 0 &  & & \\
& 1 & 0 & & \\
&  & \ddots & \ddots & \\
&  & & 1 & 0 
\end{array}
\right) \in \mathfrak{gl}_{m}(\C).
$}
$$
We can easily compute 
$\dim \OO_{\nu} = d^{2} - \sum_{i \ge 1} (2i-1)\nu_{i}$
and 
$\codim \OO_{\nu} = 
\dim \mathcal{N}^{d} - \dim \OO_{\nu} 
= \sum_{i \ge 1} (2i-1)\nu_{i} -d$,
both of which are even numbers.

We can consider the $GL_{d}(\C) \times \C^{\times}$-equivariant 
perverse coherent sheaves on the nilpotent cone $\Nn^{d}$. 

\begin{Def}
A $GL_{d}(\C) \times \C^{\times}$-equivariant perverse coherent sheaf on $\Nn^{d}$ is 
an object $\mathcal{F} \in D_{coh}^{GL_{d}(\C) \times \C^{\times}}(\Nn^{d})$ such that
for every orbit $j_{\nu} \colon  \OO_{\nu} \hookrightarrow \Nn^{d}$
\begin{enumerate}
\item $j_{\nu}^{*}\mathcal{F} \in D_{qcoh}^{GL_{d}(\C) \times \C^{\times}}(\Nn^{d})$
is supported in degrees $\le \frac{1}{2} \codim \OO_{\nu}$;
%\footnote{
%The perversity function has been changed from the previous version ($= - \dim \OO_{\nu}$).
%This new convention is the same as in \cite{Bez06} and \cite{AH18}.}
\item $j_{\nu}^{!}\mathcal{F} \in D_{qcoh}^{GL_{d}(\C) \times \C^{\times}}(\Nn^{d})$
is supported in degrees $\ge \frac{1}{2} \codim \OO_{\nu}$.
\end{enumerate}
We denote by $\Perv_{coh}^{GL_{d}(\C) \times \C^{\times}}(\Nn^{d}) \subset D_{coh}^{GL_{d}(\C)\times \C^{\times}}(\Nn^{d})$ 
the full subcategory of perverse coherent sheaves.
\end{Def} 

The simple perverse coherent sheaves
are parametrized by the set 
$$
\mathbf{O}_{d} := \{ (\nu, V)
\mid
\nu \in P(d), \,V
\in \Irr \Stab_{GL_{d}(\C)}(J_{\nu})
\}
$$
up to isomorphism and $\C^{\times}$-equivariant twist in the following way
(cf.~\cite[Section 3]{AH18}).
%Here $\Stab_{GL_{d}(\C)}(J_{\nu})$ is the centralizer of $J_{\nu}$.
For each partition $\nu = (\nu_{1}, \nu_{2}, \ldots) \in P(d)$, we define
a homomorphism $\phi_{\nu}\colon \C^{\times} \to GL_{d}(\C)$ by
$$
\phi_{\nu}(a) := \phi_{\nu_{1}}(a) \oplus \phi_{\nu_{2}}(a) \oplus \cdots, 
$$
where 
$ \phi_{m}(a) := \mathrm{diag} (a^{2(-m+1)}, a^{2(-m+3)}, \ldots, a^{2(m-1)}) \in GL_{m}(\C)$.
The homomorphism $\phi_{\nu}$ is a cocharacter associated to the nilpotent element $J_{\nu}$
in the sense of \cite[Section 5.3]{Jan04}.
In particular, we have $\mathrm{Ad}(\phi_{\nu}(a)) J_{\nu} = a^{4} J_{\nu}$,
and the group 
$$
\Stab_{GL_{d}(\C)}^{red}(J_{\nu}) := 
\{
g \in \Stab_{GL_{d}(\C)}(J_{\nu}) \mid
g \phi_{\nu}(a) = \phi_{\nu}(a) g, \, \forall a \in \C^{\times}
\}
$$
is a Levi subgroup of $\Stab_{GL_{d}(\C)}(J_{\nu})$. 
Then the image of the group embedding
$$
\Stab_{GL_{d}(\C)}^{red}(J_{\nu}) \times \C^{\times} 
\hookrightarrow
\Stab_{GL_{d}(\C) \times \C^{\times}}(J_{\nu}) ;
\quad
(g, a) \mapsto (g \phi_{\nu}(a), a)
$$
is a Levi subgroup.
Via this embedding, we make an identification
$$
\Irr (\Stab_{GL_{d}(\C) \times \C^{\times}}(J_{\nu}))
=
\Irr (\Stab_{GL_{d}(\C)}^{red}(J_{\nu}) \times \C^{\times} )
\supset 
\Irr \Stab_{GL_{d}(\C)}^{red}(J_\nu),
$$
where the set $\Irr (\Stab_{GL_{d}(\C)}^{red}(J_{\nu}))$ is regarded 
as a subset of $\Irr (\Stab_{GL_{d}(\C)}^{red}(J_{\nu}) \times \C^{\times})$
consisting of representations with the trivial $\C^{\times}$-actions.
For a pair $(\nu, V) \in \mathbf{O}_{d}$,  
let $\mathcal{V}$ be the simple 
$GL_{d}(\C) \times \C^{\times}$-equivariant vector bundle on $\OO_{\nu}$
whose fiber at $J_\nu$ is isomorphic to $V$ 
as a representation of $\Stab_{GL_{d}(\C) \rtimes \C^{\times}}(J_\nu)$. 
We define the simple perverse coherent sheaf $\Cc_{\nu,V}$ as
the following (coherent) $\IC$-extension 
$$
\Cc_{\nu, V} := (j_{\nu})_{!*} \mathcal{V} \langle \codim \OO_\nu /2 \rangle,   
$$  
where $j_{\nu} : \OO_{\nu} \hookrightarrow \overline{\OO}_{\nu}$ is the inclusion. 

Under the above notation, we have a bijection
$$
\mathbf{O}_{d} \times \Z 
\overset{1:1}{\longleftrightarrow}
\Irr \Perv_{coh}^{GL_{d}(\C) \times \C^{\times}}(\mathcal{N}^{d});
\quad
((\nu, V), m)
\longleftrightarrow
\Cc_{\nu, V}\{m/2\}.
$$ 
Thus the set $\{ [\Cc_{\nu, V}] \mid (\nu, V) \in \mathbf{O}_{d}\}$
forms a $\Z[q^{\pm 1/2}]$-basis of $K^{GL_{d}(\C) \times \C^{\times}}(\Nn^{d})$. 

Next we introduce another basis of $K^{GL_{d}(\C) \times \C^{\times}}(\Nn^{d})$. 
Let $B_{d} \subset GL_{d}(\C)$ be the Borel subgroup consisting of 
invertible {\em lower} triangular matrices and 
$\mathcal{B}_{d} := GL_{d}(\C)/B_{d}$ be the flag variety.
The cotangent bundle $T^{*} \mathcal{B}_{d}$ is naturally identified with
the space $GL_{d}(\C) \times^{B_{d}} \mathfrak{n}_{d}$, where 
$\mathfrak{n}_{d} \subset \mathfrak{gl}_{d}(\C)$ is the Lie algebra of 
strictly lower triangular matrices (= the nilpotent radical of $\mathrm{Lie}( B_{d})$).
Let $\pi \colon  T^{*} \mathcal{B}_{d} \to \mathcal{B}_{d}$ denote the natural projection  $[g, x] \mapsto [g]$ 
and $\Sp \colon  T^{*} \mathcal{B}_{d} \to \mathcal{N}^{d}$ denote
the Springer resolution $[g, x] \mapsto \mathrm{Ad}(g)x$.
A natural left $GL_{d}(\C) \times \C^{\times}$-action on 
$T^{*} \mathcal{B}_{d}$ is given by
$(h, a) \cdot [g, x] := [hg, a^{-4}x]$.
Both morphisms $\pi$ and $\Sp$ are $GL_{d}(\C) \times \C^{\times}$-equivariant.    

Let $T_{d} \subset B_{d}$ be the maximal torus 
consisting of diagonal matrices.
As before, 
the weight lattice $X := \Hom(T_{d}, \C^{\times}) = \Hom(B_{d}, \C^{\times})$ is 
identified with $\Z^{d}$.
For any $\lambda \in X$, we denote by $\Oo_{\mathcal{B}_{d}}(\lambda)$
the corresponding line bundle $GL_{d}(\C) \times^{B_{d}} \lambda$ on $\mathcal{B}_{d}$, which is regarded 
as a $GL_{d}(\C) \times \C^{\times}$-equivariant bundle with the trivial $\C^{\times}$-action.
We define the corresponding {\em Andersen-Jantzen sheaf} $AJ(\lambda)$ by
$$
AJ(\lambda) := \Sp_{*} \pi^{*} \Oo_{\mathcal{B}_{d}}(\lambda).
$$
More precisely, $\Sp_{*}$ denotes the derived push forward and
the Andersen-Jantzen sheaf $AJ(\lambda)$ is an object of $D_{coh}^{GL_{d}(\C) \times \C^{\times}}(\Nn^{d})$
(which may or may not be a genuine sheaf).

As a convention, we regard the weights of $\mathfrak{n}_{d}$ as the negative roots.
Then the set of dominant weights is 
$X_{+} = \{ \lambda = (\ell_{1}, \ldots, \ell_{d}) \in X \mid \ell_{1} \ge \cdots \ge \ell_{d} \}$.
 
For each dominant weight $\lambda \in X_{+}$, we define
$$
\Delta_{\lambda} := AJ(w_{0}\lambda)\{\delta_{\lambda}\}, \qquad
\nabla_{\lambda} := AJ(\lambda) \{ -\delta_{\lambda}\}
$$
where 
$w_{0}$ is the longest element of the Weyl group $\SG_{d}$ of $GL_{d}(\C)$ and
$\delta_{\lambda} := \min\{ \ell(w) \mid w \in \SG_{d},\ w \lambda  \in  -X_{+} \}$. 
Explicitly we have 
\begin{equation} \label{Eq:delta}
\delta_{\lambda} = \frac{1}{2} ( d(d-1)  - \sum_{k} m_{k}(m_{k}-1)),
\end{equation}
where $m_{k}$ is the multiplicity of $k \in \Z$ in the sequence $\lambda \in \Z^{d}$. 

It is known that both objects $\Delta_{\lambda}$ and $\nabla_{\lambda}$ are perverse coherent sheaves.
Indeed the family $\{ \nabla_{\lambda} \mid \lambda \in X_{+}\}$ forms a quasi-exceptional set 
of the category $D^{GL_{d}(\C) \times \C^{\times}}_{coh}(\Nn^{d})$ 
with $\{ \Delta_{\lambda} \mid \lambda \in X_{+}\}$ being its dual, 
which yields the above perverse $t$-structure (cf.~\cite{Bez03}).
In particular, there is a canonical morphism $\Delta_{\lambda} \to \nabla_{\lambda}$ for each $\lambda \in X_{+}$. 
We denote the image of this canonical morphism by $\CC_\lambda$,
which is a simple perverse coherent sheaf. 
The following result due to Achar-Hardesty~\cite{AH18} is the graded (or $\C^{\times}$-equivariant) version of the {\em Lusztig-Vogan bijection}.
The non-graded version was originally established by \cite{Ach01} (for $GL_{d}$) and 
\cite{Bez03} (for a general reductive group instead of $GL_{d}$).

\begin{Thm}[{\cite[Theorem 4.5]{AH18}}]
 \label{Thm:LV bijection}
There is a bijection 
$$\LV\colon X_{+} \xrightarrow{\sim} \mathbf{O}_{d}$$ 
such that
we have 
$
\CC_\lambda \cong 
\Cc_{\LV(\lambda)}
$
for any $\lambda \in X_{+}$.
\end{Thm}
%In particular, the set $\{ [C(\lambda)] \mid \lambda \in X_{+} \}$ forms 
%a $\Z[q^{\pm 1/2}]$-basis of $K^{GL_{d}(\C) \times \C}(\Nn^{d})$. 

We define the modified Grothendieck-Serre duality functor $\mathbb{D}_{\Nn^{d}}$ 
on $\Nn^{d}$ by
$$
\mathbb{D}_{\Nn^{d}} := \mathbb{R} \sHom(-, \mathcal{O}_{\Nn^d}).
$$ 

\begin{Rem}
 \label{Rem:canonical sheaf Nn}
The usual Grothendieck-Serre duality is defined by using
the dualizing complex. 
The dualizing complex $\omega_{\Nn^{d}}$ of the nilpotent cone $\Nn^{d}$ is
$\Oo_{\Nn^{d}} \langle d(d-1) \rangle$ (see \cite[Proposition 2.4]{AH18}). 
\end{Rem}

Let $\sigma$ be an involution of the group $GL_{d}(\C) \times \C^{\times}$ given by 
$(h, a) \mapsto (\T h^{-1}, a)$. Then the transpose map 
$x \mapsto \T x$ induces a 
$GL_{d}(\C) \times \C^{\times}$-equivariant isomorphism
$\tau\colon \Nn^{d} \xrightarrow{\sim} (\Nn^{d})^{\sigma}$.
We define an involutive auto-equivalence $\iota$ of $\Perv^{GL_{d}(\C) \times \C^{\times}}_{coh}(\Nn^{d})$
by 
$$\iota(\mathcal{F}) := \mathbb{D}_{\Nn^{d}} \circ \tau^{*}(\mathcal{F}^{\sigma}).
$$
Then we have $\iota(q^{1/2}) = q^{-1/2}$ at the level of Grothendieck group.
The following theorem was originally conjectured by Ostrik~\cite{Ost00}
and proved by Bezrukavnikov (see \cite[Introduction]{Bez03}).

\begin{Thm}[Bezrukavnikov]
 \label{Thm:Bezrukavnikov}
The $\Z[q^{\pm 1/2}]$-basis $\{ [\CC_\lambda] \mid \lambda \in X_{+} \}$ of
$K^{GL_{d}(\C) \times \C^{\times}}(\Nn^d)$
is characterized by the following properties:
\begin{itemize}
\item[(1)] $\iota[\CC_\lambda] = [\CC_\lambda]$;
\item[(2)] $[\CC_\lambda] \in [\nabla_{\lambda}] + \sum_{\lambda^{\prime} \in X_{+}} q^{-1} \Z[q^{-1}] [\nabla_{\lambda^{\prime}}]$.
\end{itemize}
\end{Thm} 

%%%%%%%%%%%%%%%%%%%%%%%%%%%%%%%%%%%%%%%%%%%%%%%%%%%%%%%%%%%%%%%%%%%%%%%%%%%%%%%%%%%%%%%%%%%%%%%%%%%%%%%%%%%%%%%%%%%%%%%%%%%%%%%%%%

\subsection{Comparison with the nilpotent cone}  \label{Ssec:comparison}

Towards a proof of Theorem~\ref{Thm:1}, let us compare 
$\Perv_{coh}^{GL_{n}(\Oo) \rtimes \C^{\times}}(\Gr_{GL_{n}})$ with 
$\Perv_{coh}^{GL_{d}(\C) \times \C^{\times}}(\mathcal{N}^{d})$. 
Fix two positive integers $n, d \in \N$ and  
consider the Schubert variety
$$
\Gr_{n}^{d} := \overline{\Gr}{}_{GL_{n}}^{d \omega_{1}} = \{ 
L \subset L_{0} \mid \dim(L_{0} / L) = d
\}.
$$
This is
a finite union of $GL_{n}(\Oo)$-orbits $\Gr_{GL_{n}}^{\nu}$
where $\nu$ runs over the set  
$$
P_{n}(d)
:=
\{ \nu =
(\nu_{1}, \ldots, \nu_{n}) \in \Z^{n} \mid
\nu_{1} \ge \cdots \ge \nu_{n} \ge 0, \,  
\nu_{1} + \cdots + \nu_{n} = d
\}
$$
of partitions of $d$ of length $\le n$, 
regarded as dominant coweights of $GL_{n}$ in the same way as before.
Let $\Dom_{n, d}$ be the set of dominant pairs $(\nu, \mu) \in \Dom_{n}$ with $\nu \in P_{n}(d)$.
Then the set $\{ [\Pp_{\nu, \mu}] \mid (\nu, \mu) \in \Dom_{n,d} \}$ forms 
a $\Z[q^{\pm 1/2}]$-basis of $K^{GL_{n}(\Oo) \rtimes \C^{\times}}(\Gr_{n}^{d})$.

We define the modified Grothendieck-Serre duality functor $\mathbb{D}_{\Gr_{n}^{d}}$ 
on $\Gr_{n}^{d}$ by
$$
\mathbb{D}_{\Gr_{n}^{d}} := \mathbb{R} \sHom(-, \mathcal{O}_{\Gr_{n}^{d}}) \langle d(n-1) \rangle.
$$ 

\begin{Rem}
 \label{Rem:canonical sheaf Gr}
The dualizing complex of the Schubert variety 
$\Gr_{n}^{d}$ is isomorphic to 
$$
\Oo_{\Gr_{GL_n}}(-n) \otimes \mathrm{det}(L_0 / t L_0)^{d}\{d(n-d)\} \langle d(n-1) \rangle
$$
(see \cite[Lemma 6.20]{CW18}).
Therefore we have 
\begin{equation} 
 \label{Eq:duality Gr}
\mathbb{D}_{\Gr_{n}^{d}}(\mathcal{F}) =  \mathbb{D} \circ \mathbb{L}(\mathcal{F})
\end{equation}
for any $\mathcal{F} \in D^{GL_{n}(\Oo) \rtimes \C^{\times}}_{coh}(\Gr_{n}^{d})$
(see Section~\ref{Ssec:CW} for the definition of $\mathbb{L}$).
\end{Rem}

Under the above notation,
we have the following morphism of quotient stacks:
$$
\psi\colon  
[(GL_{n}(\Oo) \rtimes \C^{\times}) \backslash \Gr_{n}^{d}] \to 
[(GL_{d}(\C) \times \C^{\times}) \backslash \Nn^{d}]; 
\quad
(L \subset L_{0})
\mapsto t |_{L_{0} / L}.
$$

\begin{Lem}
  \label{formally smooth}
The morphism $\psi$ is formally smooth.
\end{Lem}

\begin{proof}
  For a fixed $d$ and $N\gg0,\ \Gr_{n}^{d}=\{L : L_0\supset L\supset t^NL_0\}$.
  We have an evident morphism
  $[(GL_{n}(\Oo/t^N\Oo) \rtimes \C^{\times}) \backslash \Gr_{n}^{d}]\to
  [(GL_{n}(\Oo) \rtimes \C^{\times}) \backslash \Gr_{n}^{d}]$, and by an abuse
  of notation we will denote by
  $\psi\colon[(GL_{n}(\Oo/t^N\Oo) \rtimes \C^{\times}) \backslash \Gr_{n}^{d}] \to 
  [(GL_{d}(\C) \times \C^{\times}) \backslash \Nn^{d}]$ the composition of the former $\psi$
  with the above evident morphism. It suffices to prove that the new $\psi$ is smooth.
  Moreover, we will keep the same notation $\psi$ for the similar morphism
  $[GL_{n}(\Oo/t^N\Oo) \backslash \Gr_{n}^{d}] \to 
  [GL_{d}(\C) \backslash \Nn^{d}]$ (disregarding the extra $\C^\times$-equivariance).
  It suffices to prove that the latter $\psi$ is smooth.
  
  Given an affine test scheme $S=\Spec A$ along with its nilpotent extension 
$\wt{S}=\Spec\wt{A},\ A=\wt{A}/I$, and a morphism $\ul{\varphi}\colon S\to
[GL_{n}(\Oo/t^N\Oo)\backslash \Gr_{n}^{d}]$ along with
an extension
\begin{equation*}
  \wt{\varphi}\colon \wt{S}\to[GL_{d}(\C)\backslash \Nn^{d}]\
\on{of}\ \varphi:=\psi\circ\ul{\varphi}\colon S\to
[GL_{d}(\C)\backslash \Nn^{d}]
\end{equation*}
we have to find an extension $\ul{\wt\varphi}\colon\wt{S}\to
[GL_{n}(\Oo/t^N\Oo)\backslash \Gr_{n}^{d}]$.

We may and will assume that $A$ and $\wt{A}$ are local, hence the projective
modules are free. Then an $S$-point $\varphi$ is a free $A$-module $M$
of rank $d$ with a nilpotent endomorphism $\st\in\End_A(M)$. Similarly, an
$\wt{S}$-point $\wt\varphi$ is a free $\wt{A}$-module $\wt{M}$ of rank $d$
with a nilpotent endomorphism $\wt\st\in\End_{\wt{A}}(\wt{M})$.
An $S$-point $\ul{\varphi}$ is a free $A$-module $\ul{M}$ of rank
$nN$ with a nilpotent endomorphism $\ul\st$ ``of Jordan type $N^n$'' and
a $\ul\st$-invariant (locally) free $A$-submodule $\ul{M}'\subset\ul{M}$ such that
the quotient $\ul{M}/\ul{M}'$ is free of rank $d$.
Finally, an $\wt{S}$-point $\ul{\wt\varphi}$ is a free $\wt A$-module
$\ul{\wt M}$ of rank
$nN$ with a nilpotent endomorphism $\ul{\wt\st}$ ``of Jordan type $N^n$'' and
a $\ul{\wt\st}$-invariant (locally) free $\wt A$-submodule
$\ul{\wt M}{}'\subset\ul{\wt M}$ such that
the quotient $\ul{\wt M}/\ul{\wt M}{}'$ is free of rank $d$.

We have to prove that given $(\wt{M},\wt\st)$ and $(\ul{M}'\subset\ul{M},\ul\st)$
as above giving rise
to the same $(\wt{M}/I,\wt\st\pmod{I})=(M,\st)=(\ul{M}/\ul{M}',\ul\st\pmod{\ul{M}'})$,
there exists $(\ul{\wt M}{}'\subset\ul{\wt M},\ul{\wt\st})$ as above such that
$(\ul{\wt M}{}'\pmod{I}\subset\ul{\wt M}\pmod{I},\ul{\wt\st}\pmod{I})=
(\ul{M}'\subset\ul{M},\ul\st)$, while
$(\ul{\wt M}/\ul{\wt M}{}',\ul{\wt\st}\pmod{\ul{\wt M}{}'})=(\wt{M},\wt\st)$.
If we disregard the nilpotent operators, then the existence of the desired
extension $\ul{\wt M}{}'\subset\ul{\wt M}$ follows from the smoothness of
the evident morphism $GL_{nN}\backslash\Gr(d,nN)\to GL_{d}\backslash\on{pt}$
(and is evident by itself).

So it remains to prove that the sequence
\begin{equation*}
\End_{\wt A}(\ul{\wt M}{}'\subset\ul{\wt M})\to\End_{\wt A}(\wt{M})\oplus
\End_A(\ul{M}'\subset\ul{M})\to\End_A(M)
\end{equation*}
(see the diagram below) is exact in the middle term:

\begin{equation*}
  \begin{CD}
\End_{\wt A}(\wt{M}) @<<< \End_{\wt A}(\ul{\wt M}{}'\subset\ul{\wt M})\\
  @VVV @VVV \\
  \End_A(M) @<<< \End_A(\ul{M}'\subset\ul{M}).
  \end{CD}
\end{equation*}
This is clear since all our modules are free, and moreover, we can find a complementary
free submodule $\ul{\wt M}{}''\subset\ul{\wt M}$ such that
$\ul{\wt M}=\ul{\wt M}{}''\oplus\ul{\wt M}{}'$.
\end{proof}

\begin{Cor}
  \label{faithfully flat}
  The morphism $\psi\colon  
[(GL_{n}(\Oo) \rtimes \C^{\times}) \backslash \Gr_{n}^{d}] \to 
[(GL_{d}(\C) \times \C^{\times}) \backslash \Nn^{d}]$
is flat.
\end{Cor}

\begin{proof}
  This is~\cite[Th\'eor\`eme~17.5.1]{ega}.
\end{proof}
  
Therefore the pull-back along $\psi^{*}$ induces a triangulated functor
$$
\psi^{*}\colon D^{GL_{d}(\C) \times \C^{\times}}_{coh}(\Nn^{d}) \to 
D^{GL_{n}(\Oo) \rtimes \C^{\times}}_{coh}(\Gr_{n}^{d}).
$$  

In what follows, we restrict ourselves to the open subvariety 
$$\Nn_{n}^{d} := \bigsqcup_{\nu \in P_{n}(d)} \OO_{\nu} \subset \Nn^{d}.$$
Let $\mathbf{O}_{n, d}$ be the set of pairs $(\nu, V) \in \mathbf{O}_{d}$ with $\nu \in P_{n}(d)$. 
Then the set $\{ [\Cc_{\nu, V}] \mid (\nu, V) \in \mathbf{O}_{n,d}\}$ forms a $\Z[q^{\pm 1/2}]$-basis
of $K^{GL_{d}(\C) \rtimes \C^{\times}}(\Nn_{n}^{d})$.
We will keep the same notation $AJ(\lambda), \Delta_{\lambda}, \nabla_{\lambda}$
for their restrictions to $\Nn_{n}^{d}$.

By construction, the morphism $\psi$ has its image in the open substack $[ (GL_{d}(\C) \times \C^{\times}) \backslash \Nn_{n}^{d}]$.
More precisely, for each $\nu \in P_{n}(d)$, the morphism $\psi$ sends the $GL_{n}(\Oo)$-orbit
$\Gr_{GL_{n}}^{\nu}$ to the $GL_{d}(\C)$-orbit $\OO_{\nu}$.
Thus we have obtained the triangulated functor:
$$
\psi^{*}\colon D^{GL_{d}(\C) \times \C^{\times}}_{coh}(\Nn^{d}_{n}) \to 
D^{GL_{n}(\Oo) \rtimes \C^{\times}}_{coh}(\Gr_{n}^{d}).
$$  

\begin{Def}
We define the triangulated functor 
$$\Psi\colon D^{GL_{d}(\C) \times \C^{\times}}_{coh}(\Nn^{d}_{n}) \to 
D^{GL_{n}(\Oo) \rtimes \C^{\times}}_{coh}(\Gr_{n}^{d}) \quad
\text{by  
$
\Psi (-) := \psi^{*}(-) \langle d(n-1)/2 \rangle.  
$}
$$
\end{Def}

\begin{Prop}
 \label{Prop:Psi}
The functor $\Psi$ satisfies the following properties:
\begin{enumerate}
\item
 \label{Prop:Psi:t-exact}
$\Psi$ is
$t$-exact with respect to 
the perverse $t$-structures of both sides. Therefore it induces an exact functor 
between abelian categories:
$$
\Psi\colon 
\Perv_{coh}^{GL_{d}(\C) \times \C^{\times}}(\mathcal{N}_{n}^{d}) \to
\Perv_{coh}^{GL_{n}(\Oo) \rtimes \C^{\times}}(\Gr_{n}^{d});
$$
\item
 \label{Prop:Psi:minimal extension}
$\Psi$ is compatible with the $\IC$-extensions, i.e.~for any $\nu \in P_{n}(d)$, 
we have $$\Psi \circ (j_{\nu})_{!*} \simeq (i_{\nu})_{!*} \circ \Psi;$$
\item \label{Prop:Psi:duality}
$\Psi$ is compatible with the duality functors, i.e.~we have
$$
\Psi \circ \mathbb{D}_{\Nn^{d}_{n}} \simeq \mathbb{D}_{\Gr_{n}^{d}} \circ \Psi.  
$$
\end{enumerate} 
\end{Prop}

\begin{proof}
Note that the cohomological degree shift $[d(n-1)/2]$ in the definition of $\Psi$ arises from the fact that
$d(n-1) = \codim \OO_{\nu} + \dim \Gr_{GL_{n}}^{\nu}$ for any $\nu \in P_{n}(d)$.

For (\ref{Prop:Psi:t-exact}), we apply \cite[Lemma 3.4]{AB10}. 
To do so,
we have to check that the morphism $\psi$ is faithfully flat and Gorenstein. 
For the faithful flatness, thanks to Lemma~\ref{formally smooth}, it suffices to show that given a local $\C$-algebra $A$ 
the morphism $\psi$ yields a surjective map $[(GL_{n}(\Oo) \rtimes \C^{\times}) \backslash \Gr_{n}^{d}] (A) \twoheadrightarrow 
[(GL_{d}(\C) \times \C^{\times}) \backslash \Nn_{n}^{d}](A)$
of the sets of $A$-points. 
(Or instead, we may show that the morphism $\psi_2 \colon \Mm_{n}^{d} 
\to \Nn_{n}^{d}$ defined below is surjective as a morphism between the schemes associated with the varieties.)
This can be proved easily by the definition of $\Nn_{n}^{d}$.
For the Gorenstein property, it suffices to show that
$\psi^{!} \Oo_{\Nn^{d}}$ is a cohomological degree shift of an invertible sheaf
(see \cite[Exercise V.9.5]{Har66}).  
Since the dualizing complex of $\Nn^{d}$ is isomorphic to 
$\Oo_{\Nn^{d}} \langle d(d-1) \rangle$ (see Remark~\ref{Rem:canonical sheaf Nn}),
the sheaf $\psi^{!} (\Oo_{\Nn^{d}})$ is isomorphic to the dualizing complex of 
$\Gr_{n}^{d}$ up to cohomological degree shift and $\C^{\times}$-equivariant twist,
which is also known to be a cohomological degree shift of an invertible sheaf (see Remark~\ref{Rem:canonical sheaf Gr}).

Now (\ref{Prop:Psi:minimal extension}) can be proved in a similar way 
by the definition of the minimal extension functors.
See \cite[Theorem 4.2]{AB10}.     

The remaining assertion (\ref{Prop:Psi:duality}) follows from Remarks~\ref{Rem:canonical sheaf Nn} \& \ref{Rem:canonical sheaf Gr}
and the fact 
$$\psi^{!} (-) = \psi^{*} (-) \otimes \omega_{\Gr_{n}^{d} / \Nn_{n}^{d}} 
= (\mathbb{L} \circ \psi^{*}) (-) \langle d(d-n) \rangle.
$$ 
See \cite[Remark on pp.~143--144]{Har66} for the first equality.
\end{proof}

For a technical reason, we will introduce an auxiliary space.  
Let $\Mm_{n}^{d}$ be the variety of pairs $(L, \gamma)$ such that:
\begin{itemize}
\item[(i)] 
$L$ is a $\Oo$-lattice of $\Kk^{n}$ such that 
$\dim L_{0} /L = d$;
\item[(ii)] 
$\gamma$ is a $\C$-linear isomorphism $\C^{d} \xrightarrow{\sim} L_{0}/L$. 
\end{itemize}
We equip the space $\Mm_{n}^{d}$
with a left action of the group $GL_{d}(\C) \times GL_{n}(\Oo) \rtimes \C^{\times}$
by
$$
(h, g(t), a) \cdot (L, \gamma) := ( (g(t), a) L, (g(t), a) \circ \gamma \circ h^{-1}),
$$
where $h \in GL_{d}(\C),\ g(t) \in GL_{n}(\Oo),\ a \in \C^{\times}$ and
in the 2nd entry of the right hand side the element $(g(t), a) \in GL_{n}(\Oo) \rtimes \C^{\times}$
is regarded as a $\C$-linear isomorphism
$L_{0}/L \xrightarrow{\sim} L_{0} / (g(t), a) L$.
Then we can consider the following diagram:
$$
\Gr_{n}^{d} \xleftarrow{\psi_{1}} \Mm_{n}^{d} 
\xrightarrow{\psi_{2}} \Nn_{n}^{d}. 
$$ 
Here the morphism $\psi_{1}$ is the first projection 
$(L, \gamma) \mapsto L$ and the morphism $\psi_{2}$ is given by 
$$
\psi_{2}(L, \gamma) := \gamma^{-1} \circ t|_{L_{0}/L} \circ \gamma. 
$$   
The morphisms $\psi_{1}$ and $\psi_{2}$ are equivariant 
with respect to the actions of the group 
$GL_{d}(\C) \times GL_{n}(\Oo) \rtimes \C^{\times}$.
Here we understand that the group $GL_{d}(\C)$ (resp.~$GL_{n}(\Oo)$) acts trivially 
on $\Gr_{n}^{d}$ (resp.~$\Nn_{n}^{d}$).
Since $\psi_{1}$ is a principal $GL_{d}(\C)$-bundle, the pull-back functor 
gives an equivalence of triangulated categories:
$$
\psi_{1}^{*}\colon D^{GL_{d}(\C) \times GL_{n}(\Oo) \rtimes \C^{\times}}_{coh}(\Mm_{n}^{d})
\xrightarrow{\sim} 
D^{GL_{n}(\Oo) \rtimes \C^{\times}}_{coh}(\Gr_{n}^{d}).  
$$ 
We fix a quasi-inverse of $\psi_{1}^{*}$ and denote it by $(\psi_{1}^{*})^{-1}$.

\begin{Lem}
There is an isomorphism of functors $\psi^{*} \simeq (\psi_{1}^{*})^{-1} \circ \psi_{2}^{*}$. 
\end{Lem}
\begin{proof}
This is obvious from the construction. 
\end{proof}

For each $\nu = (\nu_{1}, \ldots, \nu_{n}) \in P_{n}(d)$, we define 
a $\C$-linear isomorphism 
$$\gamma_{\nu}\colon \C^{d} \xrightarrow{\sim} L_{0} / (t^{\nu}L_{0}) \quad
\text{
by 
$v_{j} \mapsto t^{-(\nu_{1} + \cdots + \nu_{k-1})+j-1} u_{k} \mod t^{\nu} L_{0}$}
$$
for each $1 \le k \le n$ and $\nu_{1} + \cdots + \nu_{k-1} < j \le \nu_{1} + \cdots + \nu_{k}$,
where $\{ v_{j} \in \C^{d} \mid 1 \le j \le d \}$ is the standard $\C$-basis of $\C^{d}$ 
and $\{ u_{k} \in \Kk^{n} \mid 1 \le k \le n\}$ is the standard $\Kk$-basis of $\Kk^{n}$.  
The point $p_{\nu} := (t^{\nu} L_{0}, \gamma_{\nu}) \in \Mm_{n}^{d}$ satisfies
$\psi_{1}(p_{\nu}) = [t^{\nu}]$ and $\psi_{2}(p_{\nu}) = J_{\nu}$.
Then the natural projections
$$
GL_{n}(\Oo) \rtimes \C^{\times} \leftarrow
GL_{d}(\C) \times GL_{n}(\Oo) \rtimes \C^{\times} \rightarrow
GL_{d}(\C) \times \C^{\times}
$$
induce the homomorphisms of stabilizers
$$
\Stab_{GL_{n}(\Oo) \rtimes \C^{\times}}[t^{\nu}] \xleftarrow{\sim}
\Stab_{GL_{d}(\C) \times GL_{n}(\Oo) \rtimes \C^{\times}} (p_{\nu}) \to
\Stab_{GL_{d}(\C) \times \C^{\times}}(J_{\nu}),
$$
where the left one is an isomorphism because $\psi_{1}$ is a $GL_{d}(\C)$-bundle.
Let 
$$\rho\colon \Stab_{GL_{n}(\Oo) \rtimes \C^{\times}}[t^{\nu}] \to \Stab_{GL_{d}(\C) \times \C^{\times}}(J_{\nu})$$
be the group homomorphism obtained by composing the above two homomorphisms.
%We define a Levi subgroup of ${\Stab_{GL_{n}(\Oo)}[t^{\nu}]}$ by
%$$
%\Stab_{GL_{n}(\Oo)}^{red}[t^{\nu}] := \{ g \in GL_{n}(\C) \mid g \cdot t^{\nu} = t^{\nu} \cdot g \}. 
%$$
%This is isomorphic to $\prod_{k} GL_{m_k}(\C)$,
%where $m_k$ is the multiplicity of $k$ in the sequence 
%$\nu = (\nu_{1}, \ldots, \nu_{n})$. 
%We can easily see that 
This homomorphism $\rho$ induces an isomorphism between the subgroups
$$\rho_{1}\colon \Stab_{GL_{n}(\Oo)}^{red}[t^{\nu}] \xrightarrow{\sim} \Stab_{GL_{d}(\C)}^{red}(J_{\nu}).$$
In particular, 
the assignment $(\nu, \mu) \mapsto (\nu, V_{\mu})$ 
defines a bijection $\Dom_{n,d} \xrightarrow{\sim} \mathbf{O}_{n,d}$.
Henceforth we identify $\mathbf{O}_{n,d}$ with $\Dom_{n,d}$ via this bijection
and we write $\Cc_{\nu, \mu}$ instead of $\Cc_{\nu, V_{\mu}}$.

\begin{Lem}
 \label{Lem:image IC}
For any $(\nu, \mu) \in \Dom_{n,d}$, we have an isomorphism
$$
\Psi \left( \Cc_{\nu, \mu} \{- \langle \omega_{n}, \mu \rangle \} \right) \cong 
\Pp_{\nu, \mu}. 
$$
In particular, the functor $\Psi$ induces
 an isomorphism\footnote{see~\S\ref{erratum} for a correction of this claim.} of $K$-groups
$$
[\Psi]\colon
K^{GL_{d}(\C) \times \C^{\times}}(\mathcal{N}_{n}^{d}) \xrightarrow{\sim} 
K^{GL_{n}(\Oo) \rtimes \C^{\times}}(\Gr_{n}^{d}),
$$
which gives a bijective correspondence between the classes of simple perverse coherent sheaves.
\end{Lem}
\begin{proof}
By Proposition~\ref{Prop:Psi}(\ref{Prop:Psi:minimal extension}),
it suffices to show that the fiber at $[t^{\nu}]$ of 
$\Psi\left( \Cc_{\nu, \mu}\right)$
is isomorphic to $V_{\mu}\{  - \langle \nu - \omega_{n} , \mu \rangle - \dim \Gr_{GL_{n}}^{\nu}/2 \}$ 
as a representation of $\Stab_{GL_{n}(\Oo) \rtimes \C^{\times}}[t^{\nu}]$,
disregarding cohomological degree shift.
By construction, 
we observe that  the restriction of $\rho$ to the Levi subgroup 
$\Stab_{GL_{n}(\Oo)}^{red}[t^{\nu}] \times \C^{\times} \subset \Stab_{GL_{n}(\Oo) \rtimes \C^{\times}}[t^{\nu}]$
is given by $(g, a) \mapsto (\rho_{1}(ga^{2(\nu- \omega_{n})}) \phi_{\nu}(a), a)$.
Therefore the fiber at $[t^{\nu}]$ of 
$\psi^{*} \left( \Cc_{\nu, \mu}\right)$ 
is isomorphic to the pull-back of the representation
$V_{\mu} \{ \codim \OO_{\nu} /2 \}$ along the group homomorphism 
$\Stab_{GL_{n}(\Oo)}^{red}[t^{\nu}] \times \C^{\times} \to \Stab_{GL_{d}(\C)}^{red}(J_{\nu}) \times \C^{\times}$   
given by $(g,a) \mapsto (\rho_{1}(g a^{2(\nu - \omega_{n})}), a)$. 
After the $\C^{\times}$-equivariant twist $\{ - d(n-1)/2\}$, we obtain the desired representation.  
\end{proof}

\begin{Lem}
 \label{Lem:image AJ}
Let $\lambda = (\ell_{1}, \ldots, \ell_{d}) \in X = \Z^{d}$ be a weight of $GL_{d}(\C)$. 
Then we have an isomorphism
$$
\Psi(AJ(\lambda)\{ - \langle \omega_{d}, \lambda \rangle \}) 
\cong 
\Pp_{1, \ell_{1}} * \Pp_{1, \ell_{2}} *  \cdots * \Pp_{1, \ell_{d}}, 
$$
where $\omega_{d} := (1, \ldots, 1) \in \Z^{d}$
and hence $\langle \omega_{d}, \lambda \rangle 
= \ell_{1} + \cdots + \ell_{d}$.
\end{Lem}
\begin{proof}
Let $\wt{\Gr}{}_{n}^{d}$ be the variety of flags of $\Oo$-lattices
$
L_{\bullet} = (L_{d} \subset \cdots \subset L_{1} \subset L_{0})
$
satisfying $\dim L_{i-1}/L_{i}=1$ for $1 \le i \le d$. 
This is nothing but the convolution variety 
$\Gr_{GL_{n}}^{1} \widetilde{\times} \cdots \widetilde{\times} \Gr_{GL_{n}}^{1}$ ($d$ factors).
The multiplication morphism
$\overline{m}\colon \wt{\Gr}{}_{n}^{d} \to \Gr_{n}^{d}$ is given simply by 
$L_{\bullet} = (L_{d} \subset \cdots \subset L_{1} \subset L_{0}) \mapsto (L_{d} \subset L_{0})$. 
Then we have
$$
\Pp_{1, \ell_{1}} *  \Pp_{1, \ell_{2}} * \cdots *\Pp_{1, \ell_{d}}
= \underline{m}_{*} \left( 
\bigotimes_{k=1}^{d} (L_{k-1}/L_{k})^{\otimes \ell_{k}}
\right) \langle d(n-1)/2 \rangle \{ - \langle \omega_{d}, \lambda \rangle \},
$$
where we denote by $L_{i-1}/L_{i}$ the line bundle on $\wt{\Gr}{}_{n}^{d}$
whose fiber at $L_{\bullet}$ is 
equal to $L_{i-1}/L_{i}$ by an abuse of notation.

On the other hand, we put $\widetilde{\Nn}_{n}^{d} := \Sp^{-1}(\Nn_{n}^{d})$,
which is an open subvariety of the cotangent bundle $T^{*} \mathcal{B}_{d}$.
We identify the variety $\widetilde{\Nn}_{n}^{d}$ with the variety of pairs
$(V_{\bullet}, x)$ consisting of a complete flag 
$V_{\bullet} = (\{0\} = V_{d} \subset \cdots \subset V_{1} \subset V_{0} = \C^{d})$ 
and a nilpotent endomorphism $x \in \Nn_{n}^{d}$ satisfying
$x(V_{i-1}) \subset V_{i}$ for all $1 \le i \le d$. 
Then we have 
$$\pi^{*} \Oo_{\mathcal{B}_{d}}(\lambda) |_{\widetilde{\Nn}_{n}^{d}}
=
\bigotimes_{k=1}^{d} (V_{k-1} / V_{k})^{\otimes \ell_{k}},
$$ 
where we denote by $V_{i-1}/V_{i}$ the line bundle on $\widetilde{\Nn}_{n}^{d}$
whose fiber at $(V_{\bullet}, x)$ is 
equal to $V_{i-1}/V_{i}$ by an abuse of notation.
  
Let $\widetilde{\Mm}_{n}^{d}$ be the variety of pairs 
$(L_{\bullet}, \gamma)$ such that:
\begin{itemize}
\item[(i)] 
$
L_{\bullet} = (L_{d} \subset \cdots \subset L_{1} \subset L_{0})
$
is a flag of $\Oo$-lattices with $\dim L_{i-1}/L_{i} = 1$ for $1 \le i \le d$;
\item[(ii)] $\gamma$ is a $\C$-linear isomorphism $\C^{d} \xrightarrow{\sim} L_{0}/L_{d}$.
\end{itemize}
This space $\widetilde{\Mm}_{n}^{d}$ fits into the following commutative diagram
\begin{equation}
 \label{Diag:conv-Sp}
\xy
\xymatrix{
\wt{\Gr}{}_{n}^{d} 
\ar[d]^-{\overline{m}}
&
\widetilde{\Mm}_{n}^{d}
\ar[l]_-{\widetilde{\psi}_{1}}
\ar[r]^-{\widetilde{\psi}_{2}}
\ar[d]^-{\overline{m}^{\prime}}
&
\widetilde{\Nn}_{n}^{d}
\ar[d]^{\Sp}
\\
\Gr_{n}^{d}
&
\Mm_{n}^{d}
\ar[l]_-{\psi_{1}}
\ar[r]^-{\psi_{2}}
&
\Nn_{n}^{d},
}
\endxy
\end{equation} 
where the morphism $\wt{\psi}_{1}$ is the projection 
$(L_{\bullet}, \gamma) \mapsto L_{\bullet}$ 
and the morphism $\wt{\psi}_{2}$ is given by 
$$(L_{\bullet}, \gamma) \mapsto 
(\{ 0\} \subset \gamma^{-1}(L_{d-1} / L_{d}) \subset \cdots \subset \gamma^{-1}(L_{1}/L_{d}) \subset \C^{d}, 
\gamma^{-1} \circ t|_{L_{0}/L_{d}} \circ \gamma).
$$
All arrows in the diagram~(\ref{Diag:conv-Sp}) are
$GL_{d}(\C) \times GL_{n}(\Oo) \rtimes \C^{\times}$-equivariant.
Moreover, both the left and the right squares are Cartesian. 
The morphism $\wt{\psi}_{1}$ is a principal $GL_{d}(\C)$-bundle.   

Since there is a $GL_{d}(\C) \times GL_{n}(\Oo) \rtimes \C^{\times}$-equivariant 
isomorphism of line bundles 
$$ \gamma\colon
\wt{\psi}_{2}^{*} \left( \bigotimes_{k=1}^{d} (V_{k-1} / V_{k})^{\otimes \ell_{k}} \right)
\xrightarrow{\sim}
\wt{\psi}_{1}^{*} \left( \bigotimes_{k=1}^{d} (L_{k-1}/L_{k})^{\otimes \ell_{k}}\right), 
$$
we have
\begin{align*}
\psi^{*} \Sp_{*} \left(\bigotimes_{k=1}^{d} (V_{k-1} / V_{k})^{\otimes \ell_{k}} \right)
& \cong
(\psi_{1}^{*})^{-1} \overline{m}^{\prime}_{*} \wt{\psi}_{2}^{*} \left( \bigotimes_{k=1}^{d} (V_{k-1} / V_{k})^{\otimes \ell_{k}} \right) \\
& \cong
\overline{m}_{*} (\wt{\psi}_{1}^{*})^{-1} \wt{\psi}_{2}^{*} \left( \bigotimes_{k=1}^{d} (V_{k-1} / V_{k})^{\otimes \ell_{k}} \right) \\
& \cong
\overline{m}_{*} \left( \bigotimes_{k=1}^{d} ( L_{k-1}/L_{k})^{\otimes \ell_{k}} \right), 
\end{align*}
where we applied the smooth base change formula (cf.~\cite[Proposition 5.3.15]{CG97})
to the two Cartesian squares in the diagram~(\ref{Diag:conv-Sp}).
After the shift and twist $\langle d(n-1)/2 \rangle \{- \langle \omega_{d}, \lambda \rangle \}$, 
we obtain the conclusion.  
\end{proof}

We fix an element $\beta = d_{0} \alpha_{0} + d_{1} \alpha_{1} \in \mathsf{Q}^{+}$
(of the root system of type $\mathsf{A}^{(1)}_{1}$)
such that $d_{0} - d_{1} = d$. 
For each $\bfa = (a_{1}, \ldots, a_{2n}) \in \KP_{n}(\beta)$, 
we define an object $\Ee(\bfa) \in \Perv_{coh}^{GL_{n}(\Oo) \rtimes \C^{\times}}(\Gr_{n}^{d})$
as the following convolution product:
$$
\Ee(\bfa) := (\Pp_{1, n})^{* a_{1}} * \cdots * (\Pp_{k, n+1 -k})^{* a_{k}} * \cdots * (\Pp_{1, 1-n})^{* a_{2n}}
\{ -\textstyle \sum_{k < \ell} a_{k} a_{\ell}\}.
$$ 
Then we have $\Phi(\wt{E}(\bfa)) = [\Ee(\bfa)]$ (see (\ref{Eq:rescaled dual PBW}) and (\ref{Eq:E=P})).

To each $\bfa \in \KP_{n}(\beta)$, we attach the unique dominant weight $\lambda_{\bfa} \in X_{+}$
which contains
the integer $n+1-k$ with multiplicity $a_{k}$ for each $1 \le k \le 2n$. 

\begin{Cor}
 \label{Cor:image PBW}
For each $\bfa \in \KP_{n}(\beta)$,
we have
$$
\Psi (\nabla_{\lambda_{\bfa}}\{ -\langle \omega_{d}, \lambda_{\bfa} \rangle \})
\cong \Ee(\bfa).
$$
\end{Cor}
\begin{proof}
This is a consequence of Lemma~\ref{Lem:image AJ} and the fact
$\delta_{\lambda_{\bfa}} = \sum_{k < \ell} a_{k} a_{\ell}$, 
which follows from (\ref{Eq:delta}).
\end{proof}

Let $\omega$ denote an automorphism of the group 
$GL_{d}(\C) \times \C^{\times}$ given by 
$(h, a) \mapsto (ha^{2}, a)$.
Since $\Nn_{n}^{d} = (\Nn_{n}^{d})^{\omega}$,
the operation $\mathcal{F} \mapsto \mathcal{F}^{\omega}$
defines an auto-equivalence of $D^{GL_{d}(\C) \times C^{\times}}_{coh}(\Nn_{n}^{d})$.
%\begin{Rem}
% \label{Rem:omega}
Then we have
$$
\Cc_{\nu, \mu}^{\omega} = \Cc_{\nu, \mu} \{ -\langle \omega_{n}, \mu \rangle \},
\quad
\nabla_{\lambda}^{\omega} = \nabla_{\lambda} \{ -\langle \omega_{d}, \lambda \rangle \}  
$$ 
for $(\nu, \mu) \in \Dom_{n,d}$ and $\lambda \in X_{+}$ respectively.
%In particular, we have 
%$$\nabla(\lambda_{\bfa})^{\omega} = \nabla(\lambda_{\bfa}) \{ (\height \beta - d(2n+1))/2\}$$
%for any $\bfa \in \KP_{n}(\beta)$. 
Thus, Lemma~\ref{Lem:image IC} and Corollary~\ref{Cor:image PBW} are rewritten as:
\begin{align}
\Psi (\Cc_{\nu, \mu}^{\omega}) & \simeq \Pp_{\nu, \mu}, \label{Eq:image IC} \\
\Psi (\nabla_{\lambda_{\bfa}}^{\omega}) & \simeq \Ee(\bfa) \label{Eq:image PBW}
\end{align} 
for any $(\nu, \mu) \in \Dom_{n,d}$ and $\bfa \in \KP_{n}(\beta)$ respectively.
%\end{Rem}

\begin{Cor} \label{Cor:iota-Psi}
For any $\xi \in K^{GL_{d}(\C) \times \C^{\times}}(\Nn_{n}^{d})$, we have 
$$
\iota [\Psi] (\xi) = [\Psi] \iota(\xi^{\omega^{-2}}).
$$
In particular, $[\Psi](\xi)$ is fixed by $\iota$ if and only if 
$\xi^{\omega^{-1}}$ is fixed $\iota$.
\end{Cor}
\begin{proof}
It is enough to consider the case $\xi = [\Cc_{\nu, \mu}]$ for some $(\nu, \mu) \in \Dom_{n,d}$.   
In this case, the assertion is obvious from (\ref{Eq:image IC}). 
\end{proof}

The following assertion implies Theorem~\ref{Thm:1}.

\begin{Thm}
For each $\bfa \in \KP_{n}(\beta)$,
we have $\LV(\lambda_{\bfa}) \in \Dom_{n,d}$ and
$$
\Phi(\widetilde{B}(\bfa)) = [\Pp_{\LV({\lambda_{\bfa}})}],
$$
where $\LV$ is the Lusztig-Vogan bijection (see Theorem~\ref{Thm:LV bijection}).
\end{Thm}
\begin{proof}
Let $\mathcal{L} := \bigoplus_{(\nu, \mu) \in \Dom_{n,d}} \Z[q^{-1/2}] [\Cc_{\nu, \mu}]$ 
be a $\Z[q^{-1/2}]$-lattice of $K^{GL_{d}(\C) \times \C^{\times}}(\Nn_{n}^{d})$.
Since the elements $[\Cc_{\nu, \mu}]$ are fixed by $\iota$ (cf.~Theorem~\ref{Thm:LV bijection} and Theorem~\ref{Thm:Bezrukavnikov}), 
we have $\mathcal{L} \cap \iota(\mathcal{L}) = \bigoplus_{(\nu, \mu) \in \Dom_{n,d}} \Z[\Cc_{\nu, \mu}]$.

By Corollary~\ref{Cor:characterization} and (\ref{Eq:image PBW}), we have
$$([\Psi]^{-1} \Phi(\wt{B}(\bfa)))^{\omega^{-1}} - [\nabla_{\lambda_{\bfa}}]
\in \sum_{\bfa^{\prime} \in \KP_{n}(\beta)} q^{-1}\Z[q^{-1}][\nabla_{\lambda_{\bfa^{\prime}}}].
$$
Combining this with Theorem~\ref{Thm:Bezrukavnikov}, we obtain  
$$([\Psi]^{-1} \Phi(\wt{B}(\bfa)))^{\omega^{-1}} - [\CC_{\lambda_{\bfa}}|_{\Nn_{n}^{d}}] \in q^{-1/2} \mathcal{L}.$$
By Corollary~\ref{Cor:iota-Psi}, the element $([\Psi]^{-1} \Phi(\wt{B}(\bfa)))^{\omega^{-1}}$ is $\iota$-invariant.
Thus we have $([\Psi]^{-1} \Phi(\wt{B}(\bfa)))^{\omega^{-1}} = [\CC_{\lambda_{\bfa}}|_{\Nn_{n}^{d}}]$.

Note that 
$\CC_{\lambda_{\bfa}}|_{\Nn_{n}^{d}} = \Cc_{\LV(\lambda_{\bfa})}$ if $\LV(\lambda_{\bfa}) \in \Dom_{n,d}$,
and 
$\CC_{\lambda_{\bfa}}|_{\Nn_{n}^{d}} = 0$ otherwise.
However the latter case can not happen because we know that $[\Psi]^{-1} \Phi(\wt{B}(\bfa))$ is nonzero.
Therefore we have $\Phi(\wt{B}(\bfa)) = [\Psi(\Cc_{\LV(\lambda_{\bfa})}^{\omega})] = [\Pp_{\LV(\lambda_{\bfa})}]$.
\end{proof}

%%%%%%%%%%%%%%%%%%%%%%%%%%%%%%%%%%%%%%%%%%%%%%%%%%%%%%%%%%%%%%%%%%%%%%%%%%%%%%%%%%%%%%%%%%%%%%%%%%%%%%%%%%%%%
%%%%%%%%%%%%%%%%%%%%%%%%%%%%%%%%%%%%%%%%%%%%%%%%%%%%%%%%%%%%%%%%%%%%%%%%%%%%%%%%%%%%%%%%%%%%%%%%%%%%%%%%%%%%%
\subsection{Comparison of the bar involutions}

In this complementary subsection, we prove the following  categorical version of Corollary~\ref{Cor:iota-Psi}.  

\begin{Prop} \label{Prop:categorical iota-Psi}
For any $\mathcal{F} \in D^{GL_{d}(\C) \times \C^{\times}}_{coh}(\Nn_{n}^{d})$, we have
$$
\iota \circ \Psi (\mathcal{F})
\cong
\Psi \circ \iota (\mathcal{F}^{\omega^{-2}}).
$$
\end{Prop}

For a proof, we need to introduce some new notation.

We define an automorphism $\sigma_{1}$ of the group 
$GL_{d}(\C) \times GL_{n}(\Oo) \rtimes \C^{\times}$
by
$$
\sigma_{1} (h, g(t), a) := ( \T h^{-1} a^{-4}, \T g(t)^{-1}, a),
$$
where $\T(-)$ denotes the transpose of matrices.
We will use the same notation $\sigma_{1}$ for its restrictions 
to the subgroups $GL_{n}(\Oo) \rtimes \C^{\times}$ or 
$GL_{d}(\C) \times \C^{\times}$. 
For the subgroup $GL_{n}(\Oo) \rtimes \C^{\times}$,
this notation is consistent with $\sigma_{1}$ defined in  
Section~\ref{Ssec:CW}.
For the subgroup $GL_{d}(\C) \times \C^{\times}$,
 we have a relation
$\sigma_{1} = \sigma \circ \omega^{-2} 
= \omega^{-2} \circ \sigma$.

%Let us consider a $(GL_{n}(\Oo) \rtimes \C^{\times})$-equivariant morphism 
%$$
%\eta\colon \Gr_{n}^{d}  \to GL_{n}(\Kk) \times^{GL_{n}(\Oo)} (\Gr_{n}^{d})^{\sigma_{1}} 
%= (GL_{n}(\Kk) \rtimes C^{\times}) \times^{(GL_{n}(\Oo) \rtimes \C^{\times})} (\Gr_{n}^{d})^{\sigma_{1}}
%$$
%given by $\eta( L ) := [g(t), \T g(t) L_{0}]$ where we choose an element $g(t) \in GL_{n}(\Kk)$
%such that $L=g(t) L_{0}$,
%noticing that the coset $[g(t), \T g(t) L_{0}]$ is independent of the choice of $g(t)$.  
%Then, for any sheaf $\mathcal{F} \in D^{GL_{n}(\Oo) \rtimes \C^{\times}}_{coh}(\Gr_{n}^{d})$,
%we have $\eta^{*} (\Oo \widetilde{\boxtimes} \mathcal{F}^{\sigma_{1}}) = (\mathcal{F}^{\prime})^{*}$ 
%(see Section~\ref{Ssec:CW}).

Now we consider a morphism
$$
\eta^{\prime}\colon \Mm_{n}^{d} \to 
GL_{n}(\Kk) \times^{GL_{n}(\Oo)} (\Mm_{n}^{d})^{\sigma_{1}} 
= (GL_{n}(\Kk) \rtimes \C^{\times}) \times^{(GL_{n}(\Oo) \rtimes \C^{\times})} (\Mm_{n}^{d})^{\sigma_{1}}
$$
given by 
$$
\eta^{\prime} (L, \gamma) := [g(t), (\T g(t) L_{0}, (\T g(t) t^{-1}) \circ \T \gamma^{-1})].
$$
Here $g(t) \in GL_{n}(\Kk)$ is an element such that 
$L = g(t) L_{0}$, and 
the $\C$-linear isomorphism $\T \gamma\colon (t \T g(t)^{-1} L_{0}) / t L_{0} \xrightarrow{\sim} \C^{d}$ 
is determined so that the following diagram commutes
$$
\xy
\xymatrix{
(L_{0} / g(t) L_{0} )^{*}
\ar[r]^-{\gamma^{*}}
\ar[d]^{\cong}
&
(\C^{d})^{*} 
\ar[d]^{\cong}
\\
(t \T g(t)^{-1} L_{0}) / t L_{0} 
\ar[r]^-{\T \gamma}
&
\C^{d},
}
\endxy
$$
where the isomorphism 
$(L_{0} / g(t) L_{0} )^{*}
\xrightarrow{\cong}
(t \T g(t)^{-1} L_{0}) / t L_{0}  
$
is given by the residue pairing 
$$
\langle - , - \rangle\colon \Kk^{n} \times \Kk^{n} \to \C;
\quad
\langle v(t), w(t) \rangle := \mathop{\mathrm{Res}}_{t=0} \frac{\T v(t) \cdot w(t)}{t} \mathrm{d}t,
$$
and the isomorphism $(\C^{d})^{*} \xrightarrow{\sim} \C^{d}$ 
is given by the standard pairing $\langle v, w \rangle := \T v \cdot w$. 
We can easily check that the resulting coset $[g(t), (\T g(t) L_{0}, (\T g(t) t^{-1}) \circ \T \gamma^{-1})]$ 
is independent of the choice of $g(t) \in GL_{n}(\Kk)$
such that $L= g(t) L_{0}$.

\begin{Lem}
The morphism $\eta^{\prime}$ is 
$GL_{d}(\C) \times GL_{n}(\Oo) \rtimes \C^{\times}$-equivariant.
\end{Lem}
\begin{proof}
This is proved by the following straightforward computation.
For
$(g(t)L_{0}, \gamma) \in \Mm_{n}^{d}$ and
$(h, g_{1}(t), a) \in GL_{d}(\C) \times GL_{n}(\Oo) \rtimes \C^{\times}$,
we compute:
\begin{eqnarray*}
&& \eta^{\prime}((h, g_{1}(t), a) \cdot (g(t)L_{0}, \gamma)) \\
&= 
& \eta^{\prime}(g_{1}(t)g(a^{4}t) L_{0}, (g_{1}(t), a) \circ \gamma \circ h^{-1}) \\
&=
& [g_{1}(t)g(a^{4}t), (\T (g_{1}(t) g(a^{4}t)) L_{0}, (\T (g_{1}(t) g(a^{4}t)) t^{-1}) \circ 
\T((g_{1}(t), a) \circ \gamma \circ h^{-1})^{-1}] \\
&=
& [g_{1}(t)g(a^{4}t), (\T g(a^{4}t) \T g_{1}(t) L_{0}, (\T g(a^{4}t) \T g_{1}(t) t^{-1}) \circ 
(\T g_{1}(t)^{-1}, a) \circ \T \gamma^{-1} \circ \T h)] \\
&=
& [g_{1}(t) g(a^{4}t), (\T g(a^{4}t) L_{0}, ((\T g(a^{4}t) t^{-1} ,a) \circ \T \gamma^{-1} \circ \T h)] \\
& =
& [g_{1}(t) g(a^{4}t),  ((1,a) \T g(t) L_{0},  ((1,a) \T g(t) t^{-1} a^{4}) \circ \T \gamma^{-1} \circ \T h)] \\
& =
& [g_{1}(t) g(a^{4}t), ((1,a) \T g(t) L_{0}, ((1,a) \T g(t) t^{-1}) \circ \T \gamma^{-1} \circ (\T h a^{4}))] \\
& = 
& [g_{1}(t) g(a^{4}t), (h, 1, a) \cdot (\T g(t) L_{0}, (\T g(t) t^{-1}) \circ \T \gamma^{-1})] \\
& =
& (h, g_{1}(t), a) \cdot [g(t), (\T g(t) L_{0}, \T g(t) t^{-1} \circ \T \gamma^{-1}) ].
\end{eqnarray*}
Therefore we have $\eta^{\prime} ((h, g_{1}(t), a) \cdot (g(t) L_{0}, \gamma)) = 
(h, g_{1}(t), a) \cdot \eta^{\prime}(g(t) L_{0}, \gamma)$.
\end{proof}

Let $\psi_{2}^{\prime}\colon GL_{n}(\Kk) \times^{GL_{n}(\Oo)} (\Mm_{n}^{d})^{\sigma_{1}}  
\to (\Nn_{n}^{d})^{\sigma_{1}}$ be a morphism 
defined by the assignment
$[g(t), (L, \gamma)] \mapsto \gamma^{-1} \circ t|_{L_{0}/L} \circ \gamma.$
We can easily check that this is well-defined and 
$GL_{d}(\C) \times GL_{n}(\Oo) \rtimes \C^{\times}$-equivariant.
On the other hand, 
the transpose map $x \mapsto \T x$ induces a 
$GL_{d}(\C) \times \C^{\times}$-equivariant morphism
$\tau\colon \Nn_{n}^{d} \to (\Nn_{n}^{d})^{\sigma_{1}}$.

\begin{Lem}
The following diagram commutes:
\begin{equation}
 \label{Diag:six}
\xy
\xymatrix{
\Gr_{n}^{d}
\ar[d]^-{\eta}
&
\Mm_{n}^{d}
\ar[l]_-{\psi_{1}}
\ar[r]^-{\psi_{2}}
\ar[d]^-{\eta^{\prime}}
&
\Nn_{n}^{d}
\ar[d]^-{\tau}
\\
GL_{n}(\Kk) \times^{GL_{n}(\Oo)} (\Gr_{n}^{d})^{\sigma_{1}}
&
GL_{n}(\Kk) \times^{GL_{n}(\Oo)} (\Mm_{n}^{d})^{\sigma_{1}}
\ar[l]_-{(\mathrm{id} \times \psi_{1})}
\ar[r]^-{\psi_{2}^{\prime}}
&
(\Nn_{n}^{d})^{\sigma_{1}},
}
\endxy
\end{equation}
where all arrows are $GL_{d}(\C) \times GL_{n}(\Oo) \rtimes \C^{\times}$-equivariant.
\end{Lem}
\begin{proof}
The commutativity of the left square is obvious from the definitions.
The commutativity of the right square follows from the relation
$$
\T (\gamma^{-1} \circ t|_{L_{0}/g(t)L_{0}} \circ \gamma) 
= \T \gamma \circ t|_{t \T g(t)^{-1} L_{0} / L_{0}} \circ \T \gamma^{-1},
$$
which holds for any $(g(t) L_{0}, \gamma) \in \Mm_{n}^{d}$. 
\end{proof}

\begin{proof}[Proof of Proposition \ref{Prop:categorical iota-Psi}]
Using the commutative diagram~(\ref{Diag:six}), we have
\begin{align*} 
\eta^{*}(\Oo_{\Gr_{n}^{d}} \widetilde{\boxtimes} \psi^{*}(\mathcal{F})^{\sigma_{1}}) 
& \cong
(\psi_{1}^{*})^{-1} (\eta^{\prime})^{*}(\Oo_{\Mm_{n}^{d}} \widetilde{\boxtimes} \psi_{2}^{*}(\mathcal{F}^{\sigma_{1}})) \\
& = 
(\psi_{1}^{*})^{-1} (\eta^{\prime})^{*} (\psi_{2}^{\prime})^{*} \mathcal{F}^{\sigma_{1}} \\
& =
(\psi_{1}^{*})^{-1} \psi_{2}^{*} \tau^{*} (\mathcal{F}^{\omega^{-2}})^{\sigma} \\
& \cong
\psi^{*} \tau^{*} (\mathcal{F}^{\omega^{-2}})^{\sigma}, 
\end{align*}
and hence $ \eta^{*}(\Oo_{\Gr_{n}^{d}} \widetilde{\boxtimes} \Psi(\mathcal{F})^{\sigma_{1}})  \simeq \Psi \circ \tau^{*} (\mathcal{F}^{\omega^{-2}})^{\sigma}$.
Applying the duality functor $\mathbb{D}_{\Gr_{n}^{d}}$, we have 
$$
\iota \circ \Psi (\mathcal{F}) 
= \mathbb{D}_{\Gr_{n}^{d}} \circ \eta^{*}(\Oo_{\Gr_{n}^{d}} \widetilde{\boxtimes} \Psi(\mathcal{F})^{\sigma_{1}}) 
\cong \Psi \circ \mathbb{D}_{\Nn_{n}^{d}} \circ \tau^{*} (\mathcal{F}^{\omega^{-2}})^{\sigma} 
= \Psi \circ \iota (\mathcal{F}^{\omega^{-2}}),
$$ 
where we used Proposition~\ref{Prop:Psi}(\ref{Prop:Psi:duality}) and the relation~(\ref{Eq:duality Gr}).
\end{proof}
%\end{comment}
%%%%%%%%%%%%%%%%%%%%%%%%%%%%%%%%%%%%%%%%%%%%%%%%%%%%%%%%%%%%%%%%%%%%%%%%%%%%%%%%%%%%%%%%%%%%%%%%%%%%%%%%%%%%%%%%%%%%%%%%%%%%%%%
%%%%%%%%%%%%%%%%%%%%%%%%%%%%%%%%%%%%%%%%%%%%%%%%%%%%%%%%%%%%%%%%%%%%%%%%%%%%%%%%%%%%%%%%%%%%%%%%%%%%%%%%%%%%%%%%%%%%%%%%%%%%%%%%%%%%%%%%%%%
%%%%%%%%%%%%%%%%%%%%%%%%%%%%%%%%%%%%%%%%%%%%%%%%%%%%%%%%%%%%%%%%%%%%%%%%%%%%%%%%%%%%%%%%%%%%%%%%%%%%%%%%%%%%%%%%%%%%%%%%%%%%%%%%%%%%%%%%%%%

\section{Erratum}
\label{erratum}

We are grateful to Vasily Krylov and Sabin Cautis who have noticed
that~Lemmma~\ref{Lem:image IC} is false as stated. Namely, the map 
\[[\Psi]\colon
K^{GL_{d}(\C) \times \C^{\times}}(\mathcal{N}_{n}^{d}) \xrightarrow{}
K^{GL_{n}(\Oo) \rtimes \C^{\times}}(\Gr_{n}^{d})\] is {\em not} an isomorphism, it is only injective.
Indeed, certain parts of $\nu=(\nu_1\geq\ldots\geq\nu_n\geq0)$ can be zero, so we write
$\nu=(i^{n_i})$, $\sum_{i\geq0}n_i=n$, allowing $i=0$. Then the reductive part of the stabilizer
in $GL(d)$ is $\prod_{i\geq1}GL(n_i)$, while the reductive part of the stabilizer in
$GL(n,{\mathcal O})$ is $\prod_{i\geq0}GL(n_i)$. Hence $[\Psi]$ is not surjective in general,
when $n_0>0$.

When~Lemma~\ref{Lem:image IC} is used in the proof of the main~Theorem~\ref{Thm:0}, we actually
do not use the surjectivity of $[\Psi]$, but the following weaker statement. Given
${\mathcal P}_{\nu,\mu}\in K^{GL_{n}(\Oo) \rtimes \C^{\times}}(\Gr_{n}^{d})$, there is a positive integer
$N\gg0$ such that the convolution ${\mathcal P}_{\nu,\mu}*{\mathcal P}_{\nu',\mu'}$ lies in the
image of $[\Psi]$, where $\mu'=(0,\ldots,0)$, $\nu'=(N,\ldots,N)$. This weaker statement
clearly holds true.

%%%%%%%%%%%%%%%%%%%%%%%%%%%%%%%%%%%%%%%%%%%%%%%%%%%%%%%%%%%%%%%%%%%%%%%%%%%%%%%%%%%%%%%%%%%%%%%%%%%%%%%%%%%%%%%%%%%%%%%%%%%%%%%%%%%%%%%%%%%%%%%%%
%%%%%%%%%%%%%%%%%%%%%%%%%%%%%%%%%%%%%%%%%%%%%%%%%%%%%%%%%%%%%%%%%%%%%%%%%%%%%%%%%%%%%%%%%%%%%%%%%%%%%%%%%%%%%%%%%%%%%%%%%%%%%%%%%%%%%%%%%%%%%%%%%

\end{document}